\documentclass[11pt]{article}

\usepackage{a4wide,cite,latexsym,amsfonts,amssymb,exscale,epsfig,enumerate,hyperref}
\usepackage[centertags,sumlimits,intlimits,namelimits,reqno]{amsmath}
\usepackage{amsthm}
\usepackage{picins}

\usepackage{pstricks}
\usepackage{pst-grad}
\usepackage{pst-plot}
\usepackage[tiling]{pst-fill}

\psset{linewidth=0.3pt,dimen=middle} \psset{xunit=.70cm,yunit=0.70cm}
\newgray{whitegray}{.80}
\psset{arrowsize=1pt 5,arrowlength=0.6,arrowinset=0.7}

\usepackage{graphicx}
\usepackage[all]{xy}
\SelectTips{cm}{}

\newcommand{\Grs}{\cat{Flag}_N^*}
\newcommand{\Gr}{\cat{Flag}_N}
\newcommand{\BNC}{\cal{N}\cal{H}}

\newcommand{\HG}{H^G}
\newcommand{\GrGs}{\cat{EqFlag}_N^*}
\newcommand{\GrG}{\cat{EqFlag}_N}
\newcommand{\refequal}[1]{\xy {\ar@{=}^{#1}
(-1,0)*{};(1,0)*{}};
\endxy}

\newcommand{\onen}{\mathbf{1}_{n}}
\newcommand{\onem}{\mathbf{1}_{m}}
\newcommand{\Uq}{{\bf U}_q(\mathfrak{sl}_2)}
\newcommand{\U}{\dot{{\bf U}}}
\newcommand{\UA}{{_{\cal{A}}\dot{{\bf U}}}}
\newcommand{\Ucat}{\cal{U}}
\newcommand{\Ucatq}{\cal{U}^*}
\newcommand{\UcatD}{\dot{\cal{U}}}

\newcommand{\xsum}[2]{
  \xy
  (0,.4)*{\sum};
  (0,3.7)*{\scs #2};
  (0,-2.9)*{\scs #1};
  \endxy
}

\newcommand{\bfit}[1]{\textit{#1}}

\usepackage{fancyheadings}
\newcommand{\cat}[1]{\ensuremath{\mbox{\bfseries {\upshape {#1}}}}}

\newcommand{\BOX}{\hbox {$\sqcap$ \kern -1em $\sqcup$}}

\newcommand{\To}{\Rightarrow}

\renewcommand{\to}{\rightarrow}
\newcommand{\maps}{\colon}

\newcommand{\Res}{{\rm Res}}
\newcommand{\End}{{\rm End}}

\newcommand{\scs}{\scriptstyle}

\theoremstyle{definition}
\newtheorem{thm}{Theorem}[section]
\newtheorem{cor}[thm]{Corollary}

\newtheorem{lem}[thm]{Lemma}
\newtheorem{rem}[thm]{Remark}
\newtheorem{prop}[thm]{Proposition}
\newtheorem{defn}[thm]{Definition}

        \newcommand{\be}{\begin{equation}}
        \newcommand{\ee}{\end{equation}}
        \newcommand{\ba}{\begin{eqnarray}}
        \newcommand{\ea}{\end{eqnarray}}
        \newcommand{\ban}{\begin{eqnarray*}}
        \newcommand{\ean}{\end{eqnarray*}}
        \newcommand{\barr}{\begin{array}}
        \newcommand{\earr}{\end{array}}

\numberwithin{equation}{section}

\def\emph#1{{\sl #1\/}}

\let\hat=\widehat
\let\tilde=\widetilde

\let\phi=\varphi
\let\theta=\vartheta
\let\epsilon=\varepsilon

\usepackage{bbm}
\def\C{{\mathbbm C}}
\def\N{{\mathbbm N}}

\def\Z{{\mathbbm Z}}
\def\Q{{\mathbbm Q}}

\def\cal#1{\mathcal{#1}}%
\def\1{\mathbbm{1}}%
\def\nn{\notag}

\newcommand{\twoIu}{\xybox{%
  (-6,0)*{};
  (6,0)*{};
  (0,6)*{}="f'";
  (0,12)*{}="f";
  (-3,18)*{}="t1";
  (3,18)*{}="t2";
    (-3,0)*{}="t1'";
  (3,0)*{}="t2'";
  (0,0)*{}="b";
  "t1";"f" **\crv{(-3,14)};?(.1)*\dir{<};
  "t2";"f" **\crv{(3,14)};?(.1)*\dir{<};
  "f"+(.5,0);"f'"+(.5,0) **\dir{-};
  "f"+(-.5,0);"f'"+(-.5,0) **\dir{-};
  "t1'";"f'" **\crv{(-3,4)};
  "t2'";"f'" **\crv{(3,4)};
}}
\newcommand{\twoId}{\xybox{%
  (-6,0)*{};
  (6,0)*{};
  (0,6)*{}="f'";
  (0,12)*{}="f";
  (-3,18)*{}="t1";
  (3,18)*{}="t2";
    (-3,0)*{}="t1'";
  (3,0)*{}="t2'";
  (0,0)*{}="b";
  "t1";"f" **\crv{(-3,14)};?(.25)*\dir{>};
  "t2";"f" **\crv{(3,14)};?(.25)*\dir{>};
  "f"+(.5,0);"f'"+(.5,0) **\dir{-};
  "f"+(-.5,0);"f'"+(-.5,0) **\dir{-};
  "t1'";"f'" **\crv{(-3,4)};
  "t2'";"f'" **\crv{(3,4)};
}}

\newcommand{\bbe}[1]{\xybox{%
  (-2,0)*{};
  (2,0)*{};
  (0,0);(0,-18) **\dir{-}; ?(.5)*\dir{<}+(2.3,0)*{\scriptstyle{#1}};
}}

\newcommand{\bbf}[1]{\xybox{%
  (-2,0)*{};
  (2,0)*{};
  (0,0);(0,-18) **\dir{-}; ?(.5)*\dir{>}+(2.3,0)*{\scriptstyle{#1}};
}}
\newcommand{\bbsid}{\xybox{%
  (-2,0)*{};
  (2,0)*{};
  (0,10);(0,4) **\dir{-};
}}
\newcommand{\bbpef}[1]{\xybox{%
  (-6,0)*{};
  (6,0)*{};
  (-4,0)*{}="t1";
  (4,0)*{}="t2";
  "t1";"t2" **\crv{(-4,-6) & (4,-6)}; ?(.15)*\dir{>} ?(.9)*\dir{>} ?(.5)*\dir{}+(0,-2)*{\scriptstyle{#1}};
}}
\newcommand{\bbpfe}[1]{\xybox{%
  (-6,0)*{};
  (6,0)*{};
  (-4,0)*{}="t1";
  (4,0)*{}="t2";
  "t2";"t1" **\crv{(4,-6) & (-4,-6)}; ?(.15)*\dir{>} ?(.9)*\dir{>}
  ?(.5)*\dir{}+(0,-2)*{\scriptstyle{#1}};
}}

\newcommand{\bbcfe}[1]{\xybox{%
  (-6,0)*{};
  (6,0)*{};
  (-4,0)*{}="t1";
  (4,0)*{}="t2";
  "t1";"t2" **\crv{(-4,6) & (4,6)}; ?(.15)*\dir{>} ?(.9)*\dir{>} ?(.5)*\dir{}+(0,2)*{\scriptstyle{#1}};
}}
\newcommand{\bbcef}[1]{\xybox{%
  (-6,0)*{};
  (6,0)*{};
  (-4,0)*{}="t1";
  (4,0)*{}="t2";
  "t2";"t1" **\crv{(4,6) & (-4,6)}; ?(.15)*\dir{>} ?(.9)*\dir{>} ?(.5)*\dir{}+(0,2)*{\scriptstyle{#1}};
}}

\newcommand{\ccbub}[1]{
\xybox{%
 (-6,0)*{};
  (6,0)*{};
  (-4,0)*{}="t1";
  (4,0)*{}="t2";
  "t2";"t1" **\crv{(4,6) & (-4,6)}; ?(.05)*\dir{>} ?(1)*\dir{>};
  "t2";"t1" **\crv{(4,-6) & (-4,-6)}; ?(.3)*\dir{}+(0,0)*{\bullet}+(0,-3)*{\scs {#1}};
}}
\newcommand{\cbub}[1]{
\xybox{%
 (-6,0)*{};
  (6,0)*{};
  (-4,0)*{}="t1";
  (4,0)*{}="t2";
  "t2";"t1" **\crv{(4,6) & (-4,6)}; ?(0)*\dir{<} ?(.95)*\dir{<};
  "t2";"t1" **\crv{(4,-6) & (-4,-6)}; ?(.3)*\dir{}+(0,0)*{\bullet}+(0,-3)*{\scs {#1}};
}}

\newcommand{\FEtEF}{
\xybox{%
  (-8,0)*{};
  (8,0)*{};
  (-5,5)*{}="t1f";
  (5,5)*{}="t2f";
  (-5,-5)*{}="b1f";
  (5,-5)*{}="b2f";
  (-3,0)*{}="bf";
  (3,0)*{}="bf'";
  "t1f";"bf" **\crv{(-5,1)} ?(0)*\dir{<};
  "t2f";"bf'" **\crv{(5,1)} ?(0)*\dir{>};
  "b1f";"bf" **\crv{(-5,-1)} ?(.2)*\dir{<};
  "b2f";"bf'" **\crv{(5,-1)} ?(.3)*\dir{>};
  "bf"+(0,.5);"bf'"+(0,.5) **\dir{-};
  "bf'"+(0,-.5);"bf"+(0,-.5) **\dir{-};
  }
}
\newcommand{\EFtFE}{
\xybox{%
  (-8,0)*{};
  (8,0)*{};
  (-5,5)*{}="t1f";
  (5,5)*{}="t2f";
  (-5,-5)*{}="b1f";
  (5,-5)*{}="b2f";
  (-3,0)*{}="bf";
  (3,0)*{}="bf'";
  "t1f";"bf" **\crv{(-5,1)} ?(.35)*\dir{>};
  "t2f";"bf'" **\crv{(5,1)} ?(.35)*\dir{<};
  "b1f";"bf" **\crv{(-5,-1)} ?(0)*\dir{>};
  "b2f";"bf'" **\crv{(5,-1)} ?(0)*\dir{<};
  "bf"+(0,.5);"bf'"+(0,.5) **\dir{-};
  "bf'"+(0,-.5);"bf"+(0,-.5) **\dir{-};
}}

\newcommand{\bbdl}[1]{\xybox{%
  (2,0);(0,-8) **\crv{(2,-2)&(0,-6)}; ?(.5)*\dir{>}
}}
\newcommand{\bbdlu}[1]{\xybox{%
  (2,0);(0,-8) **\crv{(2,-2)&(0,-6)}; ?(.5)*\dir{<}
}}
\newcommand{\bbdr}[1]{\xybox{%
  (-2,0);(0,-8) **\crv{(-2,-2)&(0,-6)}; ?(.5)*\dir{>}
}}
\newcommand{\bbdru}[1]{\xybox{%
  (-2,0);(0,-8) **\crv{(-2,-2)&(0,-6)}; ?(.5)*\dir{<}
}}

\title{
 Categorified quantum sl(2) and \\
 equivariant cohomology of iterated flag varieties
}
      \author{
      Aaron D.\ Lauda \\
      { \sl \small Department of Mathematics,}\\
      { \sl \small Columbia University, New York, NY 10027, USA}
         \\
      {\tt \small email: lauda@math.columbia.edu} \\}

\begin{document}
%

\date{August 14th, 2008}

\maketitle

\begin{abstract}
A 2-category was introduced in
\href{http://arXiv.org/abs/0803.3652}{math.QA/0803.3652} that categorifies
Lusztig's integral version of quantum sl(2). Here we construct for each positive
integer $N$ a representation of this 2-category using the equivariant cohomology
of iterated flag varieties. This representation categorifies the irreducible
$(N+1)$-dimensional representation of quantum sl(2).
\end{abstract}


%
\section{Introduction}
%
The quantum group $\Uq$ is the associative algebra (with unit) over $\Q(q)$ with
generators $E$, $F$, $K$, $K^{-1}$ and relations
\begin{eqnarray}
  KK^{-1}=&1&=K^{-1}K, \label{eq_UqI}\\
  KE &=& q^2EK, \\
  KF&=&q^{-2}FK, \\
  EF-FE&=&\frac{K-K^{-1}}{q-q^{-1}}. \label{eq_UqIV}
\end{eqnarray}
Lusztig's $\U$ is the nonunital algebra obtained from $\Uq$ by adjoining a
collection of orthogonal idempotents $1_n$ indexed by the weight lattice of
$\mathfrak{sl}_2$ (see for example~\cite{Lus1}). This decomposes the algebra $\U$
into a direct sum $\bigoplus_{n,m \in \Z}1_m \U 1_n$.  In $\U$ the elements $E$
and $F$ of $\Uq$ become the collection of elements $1_{n+2}E1_{n}$ and
$1_{n}F1_{n+2}$, for $n \in \Z$.  We will simplify notation and write $E1_{n}$
and $F1_{n+2}$, with it understood that $E$ increases the subscript by 2 and $F$
decreases the subscript by 2 going from right to left. The integral form $\UA$ of
$\U$ is the $\Z[q,q^{-1}]$-subalgebra of $\U$ generated by products of divided
powers $E^{(a)}1_n :=\frac{E^a}{[a]!}1_n$ and $F^{(a)}1_n :=\frac{F^a}{[a]!}1_n$,
where $[a]!$ denotes the quantum factorial $[a]!=[a][a-1] \dots [2][1]$ with
$[m]:=\frac{q^m-q^{-m}}{q-q^{-1}}$.  Algebra $\UA$ has a canonical basis in which
the structure constants are in $\N[q,q^{-1}]$.

In a recent paper~\cite{Lau3}, the author categorified $\UA$. A 2-category
$\UcatD$ was introduced whose split Grothendieck ring $K_0(\UcatD)$ is isomorphic
to the algebra $\UA$.  Objects of $\UcatD$ correspond to weights of
$\mathfrak{sl}_2$.  Morphisms from $n$ to $m$ consist of formal direct sums of
1-morphisms of the form
\begin{equation}
 \onem\cal{E}^{\alpha_1} \cal{F}^{\beta_1}\cal{E}^{\alpha_2} \cdots
 \cal{F}^{\beta_{k-1}}\cal{E}^{\alpha_k}\cal{F}^{\beta_k}\onen\{s\}
 \end{equation}
where $m = n+2(\sum\alpha_i-\sum\beta_i)$, and $s \in \Z$.  The shift $\{s\}$
lifts the $\Z[q,q^{-1}]$-module structure of $\UA$, that is, for a 1-morphism $x$
of $\UcatD$ we have $[x\{s\}]=q^s [x]$ in $K_0(\UcatD)$.  It was shown in
\cite{Lau3} that isomorphism classes of indecomposable 1-morphisms of $\UcatD$
bijectively correspond to elements in Lusztig's canonical basis, so that every
morphism in $\UcatD$ decomposes into a direct sum of shifted copies of morphisms
lifting Lusztig's canonical basis. The idea for such a categorification goes back
to the work of Crane and Frenkel \cite{CF,Fren}.

The 2-category $\UcatD$ was shown to have a number of other desirable properties.
For example, all morphisms in $\UcatD$ have both left and right adjoints, and
there is a natural enriched hom that associates a graded abelian group to each
pair of morphisms in $\UcatD$; the graded rank of this abelian group categorifies
a semilinear form defined on $\UA$. Furthermore, the morphisms $\cal{E}^a\onen$
and $\cal{F}^a\onen$ lifting the elements $E^a1_n$ and $F^a1_n$ in $\U$ admit an
action by the nilHecke algebra $\BNC_a$. By this we mean that the abelian group
of 2-morphisms $\UcatD(\cal{E}^a\onen,\cal{E}^a\onen)$ and
$\UcatD(\cal{F}^a\onen,\cal{F}^a\onen)$ are modules over $\BNC_a$. The nilHecke
algebra is the $\Bbbk$-algebra with unit generated by $u_i$ for $1 \leq i < a$,
and pairwise commuting elements $\chi_i$, for $1 \leq i \leq a$. The generators
satisfy the relations
\[
 \begin{array}{lcl}
   u_i^2 = 0 \quad \text{($1 \leq i <a)$},
    & \quad &  u_i\chi_j = \chi_ju_i \quad \text{if $|i-j|>1$}, \\
   u_iu_{i+1}u_i = u_{i+1}u_iu_{i+1} \quad \text{$(1 \leq i <a-1)$},
   & \quad & u_i\chi_i = 1+\chi_{i+1}u_i \quad \text{($1 \leq i
  <a$)},  \\
   u_iu_j = u_ju_i \quad \text{if $|i-j|>1$},
   & \quad  &  \chi_iu_i = 1 + u_i \chi_{i+1} \quad \text{($1 \leq i \nn
  <a$)}.
 \end{array}
\]

The generators $\chi_i\in \BNC_a$ span a polynomial algebra $\Bbbk[\chi_1,\chi_2,
\ldots, \chi_a]$. The 2-morphisms in $\UcatD$ that are acted on by these $\chi_i$
are denoted by $z_{n+2(a-i)} \maps \cal{E}^a\onen \to \cal{E}^a\onen$. The
subscripts are needed for compatibility with $\cal{E}\onen$ increasing the
subscript by 2.

Using the ordinary cohomology rings of iterated flag varieties, a representation
$\Gamma_N\maps \UcatD \to \Gr$ was also constructed in \cite{Lau3}.  The
representation $\Gamma_N$ is a 2-functor mapping $\UcatD$ into a 2-category
$\Gr$. The 2-category $\Gr$ has as objects the graded cohomology rings of certain Grassmannian
varieties.  The morphisms are graded bimodules obtained from the cohomology rings of
partial flag varieties (iterated flag varieties), and the 2-morphisms are
degree zero bimodule maps.  It was already known that these iterated flag varieties
categorify irreducible representations of $\Uq$ \cite{CR,BLM,FKS}.  In
\cite{Lau3} the author showed that the 2-category $\UcatD$, categorifying
Lusztig's $\U$, acts on the 2-category $\Gr$, revealing new structure not
previously observed.

The cohomology rings of Grassmannians and iterated flag varieties are described
in terms of Chern classes of naturally associated tautological bundles. For
example, fixing a complex vector space $W$ of dimension $N$, the tautological
$k$-dimensional complex vector bundle $U_{k,N}$ on the Grassmannian $Gr(k,N)$
of complex $k$-planes in $W$ has total space consisting of pairs $(V,x)$
with $x\in V$ and $V \in Gr(k,N)$. Choose a hermitian metric on $W$. From the
orthogonal complements of the fibres $V$ of the bundle $U_{k,N}$ we can construct
an $(N-k)$-dimensional complex vector bundle $U_{N-k,N}$ with the property that
\[
 U_{k,N} \oplus U_{N-k,N} \cong I^{N},
\]
with $I$ the trivial rank 1 bundle.  The Chern classes, $x_{i}:=x_i(U_{k,N}) \in
H^{2i}(Gr(k,N))$ for $1\leq i\leq k$, and $y_{j}:=y_j(U_{N-k,N}) \in
H^{2j}(Gr(N-k,N)$ for $1\leq j \leq N-k$, then satisfy
\begin{equation} \label{eq_chern_rel}
 \left(1 +x_1t + x_2 t^2\cdots +x_k t^k\right)
  \left(1+y_1 t +y_2 t^2 +\cdots + y_{N-k}t^{N-k}\right) =1
\end{equation}
by the Whitney sum axiom for Chern classes.  Above, $t$ is a formal variable used
to keep track of homogeneous elements. Borel~\cite{Borel} showed that the
cohomology ring of this variety is given by
\begin{equation}
H^*(Gr(k,N)) = \Q[x_1,\ldots,x_k,y_1,\ldots,y_{N-k}]/I_{k,N}
\end{equation}
where $I_{k,N}$ is the ideal generated by the homogeneous terms in
\eqref{eq_chern_rel}.

Here we define new representations of $\UcatD$ using the equivariant cohomology
of iterated flag varieties.  Just like the ordinary cohomology rings, the
equivariant cohomology rings are generated by Chern classes of naturally
associated bundles.  But there are no relations on the equivariant cohomology
rings; the equivariant cohomology of the Grassmannian variety $Gr(k,N)$ described
above is given by
\begin{equation}
H^*_{GL(N)}(Gr(k,N)) = \Q[x_1,\ldots,x_k,y_1,\ldots,y_{N-k}],
\end{equation}
the relations given by the homogeneous terms of \eqref{eq_chern_rel} are no
longer imposed.  From the equivariant cohomology of iterated flag varieties we
construct a 2-category $\GrG$ and a representation $\Gamma^{G}_N\maps\UcatD \to
\GrG$ for each positive integer $N$.  In Theorem~\ref{thm_catVN} we show that
$\Gamma^G_N$ also categorifies the irreducible $(N+1)$-dimensional representation
of $\U$.

\bigskip

The nilHecke algebra $\BNC_a$ is an infinite dimensional algebra.  However, the
representations $\Gamma_N \maps \UcatD \to \Gr$ constructed from ordinary
cohomology rings map $\BNC_a$ to finite-dimensional quotients.  The 2-morphism
$z_{n+2(a-i)}$ acted on by the polynomial subalgebra $\Bbbk[\chi_1,\chi_2, \ldots,
\chi_a]$ of $\BNC_a$ are mapped by $\Gamma_N$ to nilpotent 2-morphisms in $\Gr$.
In particular, $\Gamma_N(z_{n+2(a-i)})^{N}=0$ for all $1 \leq i \leq a$. A
similar nilpotent action is built into the definition of the
$\mathfrak{sl}_2$-categorifications introduced by Chuang and Rouquier~\cite{CR}.
By restricting their attention to $\mathfrak{sl}_2$-categorifications that admit
an action by the affine or degenerate affine Hecke algebra they are able to
obtain a beautiful classification for such categorified representations of
$\mathfrak{sl}_2$.

The importance of the equivariant representations $\Gamma_N^G\maps \UcatD \to
\GrG$ introduced here is that they provide $\Uq$-categorifications where the
2-morphisms $ z_{n+2(a-i)} \maps \cal{E}^a\onen \to \cal{E}^a\onen$ {\em do not}
act nilpotently. These representations admit a faithful action of the nilHecke
algebra, while still categorifying the irreducible $(N+1)$-dimensional
representation of $\Uq$.  Thus, the representations $\GrG$ provide new insights
about what to expect from the categorified representation theory of $\Uq$.

The existence of these new 2-categorical representations $\GrG$ also serves to
reinforce the choice of axioms for the 2-category $\UcatD$.  The 2-category $\Gr$
built from the cohomology of iterated flag varieties, and the 2-category $\GrG$
built from equivariant cohomology of these varieties, both satisfy all the
relations on 2-morphisms required for a representation of $\UcatD$. This
demonstrates the robustness of the 2-category $\UcatD$.

\paragraph{Acknowledgements:} I would like to thank Mikhail Khovanov for suggesting this
project and for his helpful input.

%
\section{The 2-category $\UcatD$}
%

The 2-category $\UcatD$ was constructed in \cite{Lau3}, where it was shown that
the split Grothendieck ring $K_0(\UcatD)$ is isomorphic to the algebra $\UA$ ---
Lusztig's integral version of $\Uq$.  The 2-category $\UcatD$ is the Karoubian
envelope of a 2-category $\Ucat$ that is the restriction of a graded additive
2-category $\Ucatq$ to the degree preserving 2-morphisms.  We recall the
definition here.  The reader is referred to \cite{Lau3} for more details.

%
\subsection{The 2-category $\Ucatq$ } \label{subsec_Ucat}
%

The graded additive 2-category with translation $\Ucatq$ consists of
\begin{itemize}
  \item objects: $\bfit{n}$ for $n \in \Z$ ,
  \item 1-morphisms: formal direct sums of composites of
  \[
 \onem \cal{E}^{\alpha_1} \cal{F}^{\beta_1}\cal{E}^{\alpha_2} \cdots
 \cal{F}^{\beta_{k-1}}\cal{E}^{\alpha_k}\cal{F}^{\beta_k}\onen\{s\}
  \]
where some of the $\alpha_i$ and $\beta_j$ may be zero, $m =
n+2(\sum\alpha_i-\sum\beta_i)$, and $s \in \Z$.
  \item graded 2-morphisms\footnote{Diagrams are read from bottom to top and from right to left.
The morphism $\cal{E}$ is represented by an upward pointing arrow, and $\cal{F}$
by a downward pointing arrow. For example, $U_n \maps \cal{E}\cal{E}\onen \To
\cal{E}\cal{E}\onen \maps n \to n+4$.}
\[
\begin{array}{cccc}
 z_n & \hat{z}_n & U_n & \hat{U}_n  \\ \\
  \xy
 (0,8);(0,-8); **\dir{-} ?(.75)*\dir{>}+(2.3,0)*{\scriptstyle{}};
 (0,0)*{\txt\large{$\bullet$}};
 (4,-3)*{ \bfit{n}};
 (-6,-3)*{ \bfit{n+2}};
 (-10,0)*{};(10,0)*{};
 \endxy
  &
  \xy
 (0,8);(0,-8); **\dir{-} ?(.75)*\dir{<}+(2.3,0)*{\scriptstyle{}};
 (0,0)*{\txt\large{$\bullet$}};
 (6,-3)*{ \bfit{n+2}};
 (-4,-3)*{ \bfit{n}};
 (-10,0)*{};(10,0)*{};
 \endxy
  &    \xy 0;/r.2pc/:
    (0,0)*{\twoIu};
    (6,0)*{ \bfit{n}};
    (-8,0)*{ \bfit{n+4}};
    (-18,0)*{};(18,0)*{};
    \endxy
  &
   \xy 0;/r.2pc/:
    (0,0)*{\twoId};
    (8,0)*{ \bfit{n+4}};
    (-6,0)*{ \bfit{n}};
    (-14,0)*{};(14,0)*{};
    \endxy
\\ \\
   \;\; \text{ {\rm deg} 2}\;\;
 & \;\;\text{ {\rm deg} 2}\;\;
 & \;\;\text{ {\rm deg} -2}\;\;
  & \;\;\text{ {\rm deg} -2}\;\;
\end{array}
\]
\[
\begin{array}{ccccc}
 \eta_n & \hat{\varepsilon}_n & \hat{\eta}_n &
 \varepsilon_n \\ \\
    \xy
    (0,-3)*{\bbpef{}};
    (8,-5)*{ \bfit{n}};
    (-4,3)*{\scs \cal{F}};
    (4,3)*{\scs \cal{E}};
    (-12,0)*{};(12,0)*{};
    \endxy
  & \xy
    (0,-3)*{\bbpfe{}};
    (8,-5)*{ \bfit{n}};
    (-4,3)*{\scs \cal{E}};
    (4,3)*{\scs \cal{F}};
    (-12,0)*{};(12,0)*{};
    \endxy
  & \xy
    (0,0)*{\bbcef{}};
    (8,5)*{ \bfit{n}};
    (-4,-6)*{\scs \cal{F}};
    (4,-6)*{\scs \cal{E}};
    (-12,0)*{};(12,0)*{};
    \endxy
  & \xy
    (0,0)*{\bbcfe{}};
    (8,5)*{ \bfit{n}};
    (-4,-6)*{\scs \cal{E}};
    (4,-6)*{\scs \cal{F}};
    (-12,0)*{};(12,0)*{};
    \endxy\\ \\
  \;\;\text{ {\rm deg} n+1}\;\;
 & \;\;\text{ {\rm deg} 1-n}\;\;
 & \;\;\text{ {\rm deg} n+1}\;\;
 & \;\;\text{ {\rm deg} 1-n}\;\;
\end{array}
\]
together with identity 2-morphisms and isomorphisms $x \simeq x\{s\}$ for each
1-morphism $x$, such that:
  \item  $\mathbf{1}_{n+2}\cal{E}\onen$ and $\onen\cal{F}\mathbf{1}_{n+2}$ are
biadjoints with units and counits given by the pairs $(\eta_n,\varepsilon_{n+2})$
and $(\hat{\varepsilon}_{n},\hat{\eta}_{n-2})$.
   \item All 2-morphisms are cyclic with respect to the above biadjoint
   structure.  This is ensured by the relations,
   \begin{equation}
    \xy 0;/r.17pc/:
    (-8,5)*{}="1";
    (0,5)*{}="2";
    (0,-5)*{}="2'";
    (8,-5)*{}="3";
    (-8,-10);"1" **\dir{-};
    "2";"2'" **\dir{-} ?(.5)*\dir{<};
    "1";"2" **\crv{(-8,12) & (0,12)} ?(0)*\dir{<};
    "2'";"3" **\crv{(0,-12) & (8,-12)}?(1)*\dir{<};
    "3"; (8,10) **\dir{-};
    (15,-9)*{ \bfit{n+2}};
    (-12,9)*{ \bfit{n}};
    (0,4)*{\txt\large{$\bullet$}};
    \endxy
    \quad =
    \quad
       \xy 0;/r.17pc/:
    (-8,0)*{}="1";
    (0,0)*{}="2";
    (8,0)*{}="3";
    (0,-10);(0,10)**\dir{-} ?(.5)*\dir{<};
    (8,5)*{ \bfit{n+2}};
    (-6,5)*{ \bfit{n}};
    (0,4)*{\txt\large{$\bullet$}};
    \endxy
    \quad =
    \quad
    \xy 0;/r.17pc/:
    (8,5)*{}="1";
    (0,5)*{}="2";
    (0,-5)*{}="2'";
    (-8,-5)*{}="3";
    (8,-10);"1" **\dir{-};
    "2";"2'" **\dir{-} ?(.5)*\dir{<};
    "1";"2" **\crv{(8,12) & (0,12)} ?(0)*\dir{<};
    "2'";"3" **\crv{(0,-12) & (-8,-12)}?(1)*\dir{<};
    "3"; (-8,10) **\dir{-};
    (15,-9)*{ \bfit{n+2}};
    (-12,9)*{ \bfit{n}};
    (0,4)*{\txt\large{$\bullet$}};
    \endxy
\end{equation}
   \begin{equation}
    \xy 0;/r.17pc/:
    (-9,8)*{}="1";
    (-3,8)*{}="2";
    (-9,-16);"1" **\dir{-};
    "1";"2" **\crv{(-9,14) & (-3,14)} ?(0)*\dir{<};
    (9,-8)*{}="1";
    (3,-8)*{}="2";
    (9,16);"1" **\dir{-};
    "1";"2" **\crv{(9,-14) & (3,-14)} ?(1)*\dir{>} ?(.05)*\dir{>};
    (-15,8)*{}="1";
    (3,8)*{}="2";
    (-15,-16);"1" **\dir{-};
    "1";"2" **\crv{(-15,20) & (3,20)} ?(0)*\dir{>};
    (15,-8)*{}="1";
    (-3,-8)*{}="2";
    (15,16);"1" **\dir{-};
    "1";"2" **\crv{(15,-20) & (-3,-20)} ?(.03)*\dir{>}?(1)*\dir{>};
    (0,0)*{\twoIu};
    (24,-9)*{ \bfit{n+4}};
    (-20,9)*{ \bfit{n}};
    \endxy
    \quad =
    \quad
       \xy 0;/r.17pc/:
    (0,0)*{\twoId};
    (8,0)*{ \bfit{n+4}};
    (-6,0)*{ \bfit{n}};
    \endxy
    \quad =
    \quad
    \xy 0;/r.17pc/:
    (9,8)*{}="1";
    (3,8)*{}="2";
    (9,-16);"1" **\dir{-};
    "1";"2" **\crv{(9,14) & (3,14)} ?(0)*\dir{<};
    (-9,-8)*{}="1";
    (-3,-8)*{}="2";
    (-9,16);"1" **\dir{-};
    "1";"2" **\crv{(-9,-14) & (-3,-14)} ?(1)*\dir{>} ?(.05)*\dir{>};
    (15,8)*{}="1";
    (-3,8)*{}="2";
    (15,-16);"1" **\dir{-};
    "1";"2" **\crv{(15,20) & (-3,20)} ?(0)*\dir{<};
    (-15,-8)*{}="1";
    (3,-8)*{}="2";
    (-15,16);"1" **\dir{-};
    "1";"2" **\crv{(-15,-20) & (3,-20)} ?(.03)*\dir{>}?(1)*\dir{>};
    (0,0)*{\twoIu};
    (24,-9)*{ \bfit{n+4}};
    (-20,9)*{ \bfit{n}};
    \endxy
\end{equation}
The cyclic conditions on 2-morphisms implies that boundary preserving planar
isotopies of diagrams result in the same 2-morphism.
  \item All dotted closed bubbles of negative degree are zero. That is,
  \[
\xy
 (-12,0)*{\cbub{m}};
 (-8,8)*{\bfit{n}};
 \endxy \;=\;0, \quad \text{if $m< n-1$}, \qquad \quad
  \xy
 (-12,0)*{\ccbub{m}};
 (-8,8)*{\bfit{n}};
 \endxy \;=\;0, \quad \text{if $m< -n-1$,}
  \]
for all $m \in \Z_+$ and $n \in \Z$. Above we use the shorthand
\[
\xy 0;/r.19pc/:
    (-8,0)*{}="1";
    (0,0)*{}="2";
    (8,0)*{}="3";
    (0,-8);(0,8)**\dir{-} ?(.5)*\dir{>};
    (-7,-5)*{ \bfit{n+2}};
    (5,-5)*{ \bfit{n}};
    (0,4)*{\txt\large{$\bullet$}}+(3,0)*{\scs m};
    \endxy
    \quad = \quad
  \left(\xy 0;/r.19pc/:
    (-8,0)*{}="1";
    (0,0)*{}="2";
    (8,0)*{}="3";
    (0,-8);(0,8)**\dir{-} ?(.5)*\dir{>};
    (-7,-5)*{ \bfit{n+2}};
    (5,-5)*{ \bfit{n}};
    (0,4)*{\txt\large{$\bullet$}};
    \endxy \right)^m
\]
to denote iterated composites of the 2-morphisms $z_n$.
 \item The nilHecke algebra $\BNC_a$ acts on $\Ucatq(\cal{E}^{a}\onen,\cal{E}^{a}\onen)$ and
$\Ucatq(\cal{F}^{a}\onen,\cal{F}^{a}\onen)$ for all $n \in \Z$.
\begin{eqnarray}\label{eq_Nil_nilpotent}
  \xy 0;/r.18pc/:
  (0,-8)*{\twoIu};
  (0,8)*{\twoIu};
  (8,8)*{\bfit{n}};
 \endxy
 \qquad = \qquad 0
\end{eqnarray}
\begin{eqnarray} \label{eq_Nil_II}
  \xy 0;/r.18pc/:
  (3,9);(3,-9) **\dir{-}?(.5)*\dir{<}+(2.3,0)*{};
  (-3,9);(-3,-9) **\dir{-}?(.5)*\dir{<}+(2.3,0)*{};
  (8,2)*{\bfit{n}};
 \endxy
 \quad = \quad
  \xy 0;/r.18pc/:
  (0,0)*{\twoIu};
  (-2,-5)*{ \bullet};
  (8,2)*{\bfit{n}};
 \endxy
 \;\; - \;\;
  \xy 0;/r.18pc/:
  (0,0)*{\twoIu};
  (2,5)*{ \bullet};
  (8,2)*{\bfit{n}};
 \endxy
 \quad = \quad
  \xy 0;/r.18pc/:
  (0,0)*{\twoIu};
  (-2,5)*{ \bullet};
  (8,2)*{\bfit{n}};
 \endxy
 \;\; - \;\;
  \xy 0;/r.18pc/:
  (0,0)*{\twoIu};
  (2,-5)*{ \bullet};
  (8,2)*{\bfit{n}};
 \endxy
\end{eqnarray}
\begin{eqnarray}\label{eq_Nil_ReidemeisterIII}
 \vcenter{ \xy 0;/r.16pc/:
    (0,0)*{\twoIu};
    (6,16)*{\twoIu};
    (-3,8);(-3,24) **\dir{-}?(1)*\dir{>};
    (0,32)*{\twoIu};
    (9,-8);(9,8) **\dir{-};
    (9,24);(9,42) **\dir{-}?(1)*\dir{>};
    (14,16)*{\bfit{n}};
 \endxy}
 \quad
 =
 \quad
  \vcenter{\xy 0;/r.16pc/:
    (0,0)*{\twoIu};
    (-6,16)*{\twoIu};
    (3,8);(3,24) **\dir{-}?(1)*\dir{>};
    (0,32)*{\twoIu};
    (-9,-8);(-9,8) **\dir{-};
    (-9,24);(-9,42) **\dir{-}?(1)*\dir{>};
    (8,16)*{\bfit{n}};
 \endxy}
\end{eqnarray}
for all values of $n \in Z$.  The action on
$\Ucatq(\cal{F}^{a}\onen,\cal{F}^{a}\onen)$ is obtained from the above relations
using biadjointness.

\item The 1-morphisms $\cal{E}$ and $\cal{F}$ lift the relations of $E$ and $F$
in $\Uq$. This is ensured by requiring the equalities
\begin{eqnarray} \label{eq_reduction}
  \text{$\xy 0;/r.17pc/:
  (14,8)*{\bfit{n}};
  (0,0)*{\twoIu};
  (-3,-12)*{\bbsid};
  (-3,8)*{\bbsid};
  (3,8)*{}="t1";
  (9,8)*{}="t2";
  (3,-8)*{}="t1'";
  (9,-8)*{}="t2'";
   "t1";"t2" **\crv{(3,14) & (9, 14)};
   "t1'";"t2'" **\crv{(3,-14) & (9, -14)};
   (9,0)*{\bbf{}};
 \endxy$} \;\; = \;\; -\sum_{\ell=0}^{-n}
   \xy
  (14,8)*{\bfit{n}};
  (0,0)*{\bbe{}};
  (12,-2)*{\cbub{n-1+\ell}};
  (0,6)*{\bullet}+(5,-1)*{\scs -n-\ell};
 \endxy
\qquad \qquad
  \text{$ \xy 0;/r.17pc/:
  (-12,8)*{\bfit{n}};
  (0,0)*{\twoIu};
  (3,-12)*{\bbsid};
  (3,8)*{\bbsid};
  (-9,8)*{}="t1";
  (-3,8)*{}="t2";
  (-9,-8)*{}="t1'";
  (-3,-8)*{}="t2'";
   "t1";"t2" **\crv{(-9,14) & (-3, 14)};
   "t1'";"t2'" **\crv{(-9,-14) & (-3, -14)};
   (-9,0)*{\bbf{}};
 \endxy$} \;\; = \;\;
 \sum_{j=0}^{n}
   \xy
  (-12,8)*{\bfit{n}};
  (0,0)*{\bbe{}};
  (-12,-2)*{\ccbub{-n-1+j}};
  (0,6)*{\bullet}+(5,-1)*{\scs n-j};
 \endxy
\end{eqnarray}
\begin{eqnarray}
 \vcenter{\xy 0;/r.17pc/:
  (-8,0)*{};
  (8,0)*{};
  (-4,10)*{}="t1";
  (4,10)*{}="t2";
  (-4,-10)*{}="b1";
  (4,-10)*{}="b2";
  "t1";"b1" **\dir{-} ?(.5)*\dir{<};
  "t2";"b2" **\dir{-} ?(.5)*\dir{>};
  (10,2)*{\bfit{n}};
  (-10,2)*{\bfit{n}};
  \endxy}
&\quad = \quad&
 -\;\;
 \vcenter{\xy 0;/r.17pc/:
  (0,0)*{\FEtEF};
  (0,-10)*{\EFtFE};
  (10,2)*{\bfit{n}};
  (-10,2)*{\bfit{n}};
  \endxy}
  \quad + \quad
   \sum_{\ell=0}^{n-1} \sum_{j=0}^{\ell}
    \vcenter{\xy 0;/r.17pc/:
    (-10,10)*{\bfit{n}};
    (-8,0)*{};
  (8,0)*{};
  (-4,-15)*{}="b1";
  (4,-15)*{}="b2";
  "b2";"b1" **\crv{(5,-8) & (-5,-8)}; ?(.1)*\dir{<} ?(.9)*\dir{<}
  ?(.8)*\dir{}+(0,-.1)*{\bullet}+(-5,2)*{\scs \ell-j};
  (-4,15)*{}="t1";
  (4,15)*{}="t2";
  "t2";"t1" **\crv{(5,8) & (-5,8)}; ?(.08)*\dir{>} ?(.97)*\dir{>}
  ?(.4)*\dir{}+(0,-.2)*{\bullet}+(3,-2)*{\scs n-1-\ell};
  (0,0)*{\ccbub{\scs -n-1+j}};
  \endxy} \nn
 \\
 \vcenter{\xy 0;/r.17pc/:
  (-8,0)*{};
  (8,0)*{};
  (-4,10)*{}="t1";
  (4,10)*{}="t2";
  (-4,-10)*{}="b1";
  (4,-10)*{}="b2";
  "t1";"b1" **\dir{-} ?(.5)*\dir{>};
  "t2";"b2" **\dir{-} ?(.5)*\dir{<};
  (10,2)*{\bfit{n}};
  (-10,2)*{\bfit{n}};
  \endxy}
&\quad = \quad&
 -\;\;
 \vcenter{\xy 0;/r.17pc/:
  (0,0)*{\EFtFE};
  (0,-10)*{\FEtEF};
  (10,2)*{\bfit{n}};
  (-10,2)*{\bfit{n}};
  \endxy}
  \quad + \quad
\sum_{\ell=0}^{-n-1} \sum_{j=0}^{\ell}
    \vcenter{\xy 0;/r.17pc/:
    (-8,0)*{};
  (8,0)*{};
  (-4,-15)*{}="b1";
  (4,-15)*{}="b2";
  "b2";"b1" **\crv{(5,-8) & (-5,-8)}; ?(.1)*\dir{>} ?(.95)*\dir{>}
  ?(.8)*\dir{}+(0,-.1)*{\bullet}+(-5,2)*{\scs \ell-j};
  (-4,15)*{}="t1";
  (4,15)*{}="t2";
  "t2";"t1" **\crv{(5,8) & (-5,8)}; ?(.1)*\dir{<} ?(.9)*\dir{<}
  ?(.4)*\dir{}+(0,-.2)*{\bullet}+(3,-2)*{\scs -n-1-\ell};
  (0,0)*{\cbub{\scs n-1+j}};
  (-10,10)*{\bfit{n}};
  \endxy} \label{eq_ident_decomp}
\end{eqnarray}
for all $n \in \Z$, where the label next to the bullet indicates the number of
composites of the 2-morphisms $z_n$ and $\hat{z}_n$. In the above equations, and
throughout this paper, all summations are increasing sums, or else they are taken
to be zero. This means that in \eqref{eq_reduction} the right--hand--side of the
first equation is zero when $n>0$, and the right--hand--side of the second
equation is zero when $n<0$.
\end{itemize}

Notice that for some values of $n$ the dotted bubbles appearing in
\eqref{eq_reduction} and \eqref{eq_ident_decomp} have negative labels. A
composite of $z_n$ or $\hat{z}_n$ with itself a negative number of times does not
make sense.  These dotted bubbles with negative labels, called {\em fake
bubbles}, are formal symbols inductively defined by the equations
\[
\xy 0;/r.18pc/:
 (0,0)*{\cbub{n-1}{}};
  (4,8)*{n};
 \endxy
 \quad = \quad
  \xy 0;/r.18pc/:
 (0,0)*{\ccbub{n-1}{}};
  (4,8)*{n};
 \endxy
  \quad = \quad 1 ,
\]
\begin{center}
 \makebox[0pt]{ $
\left( \xy 0;/r.16pc/:
 (0,0)*{\ccbub{-n-1}};
  (4,8)*{\bfit{n}};
 \endxy
 +
 \xy 0;/r.16pc/:
 (0,0)*{\ccbub{-n-1+1}};
  (4,8)*{\bfit{n}};
 \endxy t
 +\xy 0;/r.16pc/:
 (0,0)*{\ccbub{-n-1+2}};
  (4,8)*{\bfit{n}};
 \endxy t^2
 + \cdots +
 \xy 0;/r.16pc/:
 (0,0)*{\ccbub{-n-1+j}};
  (4,8)*{\bfit{n}};
 \endxy t^j
 + \cdots
\right)
\left( \xy 0;/r.16pc/:
 (0,0)*{\cbub{n-1}};
  (4,8)*{\bfit{n}};
 \endxy
 +
 \xy 0;/r.16pc/:
 (0,0)*{\cbub{n-1+1}};
  (4,8)*{\bfit{n}};
 \endxy t
 + \cdots +
 \xy 0;/r.16pc/:
 (0,0)*{\cbub{n-1+j}};
 (4,8)*{\bfit{n}};
 \endxy t^j
 + \cdots
\right)  =1.$ }
\end{center}
Although the labels are negative, one can check that the overall degree of each
fake bubble is still positive, so that these fake bubbles do not violate the
positivity of dotted bubble axiom.  See \cite{Lau3} for more details.

%
\subsection{Karoubi envelope and $\UcatD$}
%

The Karoubi envelope $Kar(\cal{C})$ of a category $\cal{C}$ is an enlargement of
the category $\cal{C}$ in which all idempotents split.  There is a fully faithful
functor $\cal{C} \to Kar(\cal{C})$ that is universal with respect to functors
that split idempotents in $\cal{C}$.  This means that if $F\maps \cal{C} \to
\cal{D}$ is any functor where all idempotents split in $\cal{D}$, then $F$
extends uniquely (up to isomorphism) to a functor $\tilde{F} \maps Kar(\cal{C})
\to \cal{D}$ (see for example \cite{Bor}, Proposition 6.5.9).  Furthermore, for
any other functor $G \maps \cal{C} \to \cal{D}$ and a natural transformation
$\alpha \maps F \To G$, $\alpha$ extends uniquely to a natural transformation
$\tilde{\alpha} \maps \tilde{F}\To\tilde{G}$.

Let $\Ucat$ denote the restriction of $\Ucatq$ to those 2-morphisms that are
degree preserving.

\begin{defn}
Define the 2-category $\UcatD$ to have the same objects as $\Ucat$ and
$\UcatD(\bfit{n},\bfit{m}) = Kar\left(\Ucat(\bfit{n},\bfit{m})\right)$. The
fully-faithful functors $\Ucat(\bfit{n},\bfit{m}) \to \UcatD(\bfit{n},\bfit{m})$
combine to form a 2-functor $\Ucat \to \UcatD$ universal with respect to
splitting idempotents in the hom categories $\UcatD(\bfit{n},\bfit{m})$.  The
composition functor $\UcatD(\bfit{n},\bfit{m}) \times \UcatD(\bfit{m},\bfit{p})
\to \UcatD(\bfit{n},\bfit{p})$ is induced by the universal property of the
Karoubi envelope from the composition functor for $\Ucat$.
\end{defn}

\parpic[r]{$
 \xymatrix{
  \Ucatq \ar[dr]_{\Psi} \ar[r]^{} & \UcatD \ar@{.>}[d]^{\tilde{\Psi}} \\
  & \cal{K}
 }
$}Let $\cal{K}$ be a 2-category in which idempotents split. Any 2-functor
$\Psi\maps\Ucatq \to \cal{K}$ extends uniquely (up to isomorphism) to a
representation $\tilde{\Psi}\maps \UcatD \to \cal{K}$.  The 2-functor
$\tilde{\Psi}$ is obtained from $\Psi$ by restricting to the degree preserving
2-morphisms of $\Ucatq$ and using the universal property of the Karoubi envelope.

\begin{thm}[Theorem 9.1.3 \cite{Lau3}]  \label{thm_Groth}
The split Grothendieck ring $K_0(\UcatD)$ is isomorphic to Lusztig's $\UA$.
\end{thm}

%
\section{Iterated flag varieties}
%

For details on the equivariant cohomology of Flag varieties we recommend the
lecture notes of Fulton~\cite{Fulton2}.

Let $X$ be a space equipped with an action on the left by a Lie group $G$. The
equivariant cohomology of $X$ with respect to $G$ is the ordinary cohomology of
$EG \times^G X$:
\begin{eqnarray}
  H^*(EG \times^G X).
\end{eqnarray}
The equivariant cohomology of a point is then
\begin{eqnarray}
  H_G^*(pt) = H^*(BG,\Q).
\end{eqnarray}
The map $X \to pt$ induces on cohomology a map $H_G^*(pt) \to H^*_G(X)$ giving
the equivariant cohomology of $X$ the structure of an algebra over $H_G^*(pt)$.

Here we are interested in $G=GL(N)$ equivariant cohomology of various partial
flag varieties.  For $GL(N)$ the equivariant cohomology ring of a point is
generated by Chern classes $x_i$, $y_i$ of degree $2i$, modulo an ideal
$I_{N,\infty}$,
\begin{eqnarray}
  H_{GL(N)}^*(pt) = H^*\big(Gr(N, \infty)\big) = \Q[x_1,x_2,\ldots,x_N,y_1,y_2,
  \ldots y_j, \ldots]/I_{N,\infty} ,
\end{eqnarray}
where $I_{N,\infty}$ is the ideal generated by the homogeneous terms in
\[
 \left(
 1 + x_1 t + x_2 t^2 + \cdots x_N t^N
 \right)
 \left(
 1 + y_1 t + y_2 t^2 + \cdots y_j t^j + \cdots
 \right)
 =1.
\]
Thus, $H_{GL(N)}^*(pt)$ is isomorphic to the polynomial ring
\begin{eqnarray} \label{eq_EQpoint}
   H_{GL(N)}^*(pt) \cong \Q[x_1, x_2, \ldots, x_{N-1},x_N]
\end{eqnarray}
with $x_i$ in degree $2i$.

Choose a positive integer $N$ and let $n=2k-N$. Fix a complex vector space $W$ of
dimension $N$.  For $0 \leq k \leq N$, let $G_k$ denote the variety of complex
k-planes in $W$. In this notation we suppress the explicit dependence on $N$. If
we wish to make this dependence explicit we use the notation $Gr(k,N)$. The
equivariant cohomology ring of $G_k$ has a natural structure of a $\Z$-graded
algebra,
\[
 H^*_{GL(N)}(G_k,\Q) = \oplus_{i}H^i_{GL(N)}(G_k,\Q) ~.
\]
For simplicity we sometimes write $\HG_k:=H^*_{GL(N)}(G_k,\Q)$.

$GL(N)$ acts transitively on $G_k$, so the equivariant cohomology of $G_k$ is
\begin{eqnarray}
  H^*_{GL(N)}(Gr(k,N)) = H^*_{{\rm Stab}(pt)}\big(pt \in Gr(k,N)\big),
\end{eqnarray}
where the stabilizer of a point $\left(\C^k \subset \C^N\right)$ in $G_k$ is the
group of invertible block $\big(k \times (N-k)\big)$ upper--triangular matrices.
This group is contractible onto its subgroup $GL(k) \times GL(N-k)$.  Hence,
\begin{eqnarray}
 H_k^G & \cong & H^*_{GL(k) \times GL(N-k)}(pt) \cong  H^*_{GL(k)}(pt)
 \otimes H^*_{GL(N-k)}(pt) \nn \\
 & \cong & \Q[x_{1},x_{2}, \cdots, x_{k}] \otimes \Q[y_{1},y_{1}, \cdots, y_{N-k}]
  \nn \\
 & \cong & \Q[x_{1,n},x_{2,n}, \cdots, x_{k,n}; y_{1,n},y_{1,n}, \cdots, y_{N-k,n}]
\end{eqnarray}
where we have introduced a parameter $n=2k-N$ as an extra subscript in the last
equation, with $\deg x_{j,n}=2j$ and $\deg y_{j,n}=2j$.  The convenience of the
additional subscript is that the equivariant cohomology ring of the variety
$G_{k+1}$ of complex $(k+1)$-planes in $W$ is given by
\begin{equation}
\HG_{k+1} = \Q[x_{1,n+2},x_{2,n+2}, \cdots, x_{k+1,n+2}; y_{1,n+2},y_{1,n+2},
\cdots, y_{N-k-1,n+2}],
\end{equation}
since $2(k+1)-N=n+2$.  Thus, these two rings are easily distinguished.

For $0 \leq k < m \leq N$ let $G_{k,m}$ be the variety of partial flags
\[
 \{ (L_k,L_m)| 0 \subset L_k \subset L_m \subset W, \; \dim_{\C} L_k =k, \;
 \dim_{\C}L_m=m\} ~.
\]
We also denote this same variety by $G_{m,k}$.  Let $\HG_{k,m}$ be the
equivariant cohomology algebra of $G_{k,m}$.  Forgetful maps
\[
 \xymatrix{ G_k & G_{k,m} \ar[l]_-{p_1} \ar[r]^-{p_2} & G_m}
\]
induce maps of equivariant cohomology rings
\[
 \xymatrix{ \HG_k \ar[r]^-{p_1^*} & \HG_{k,m}   & \HG_m \ar[l]_-{p_2^*}}
\]
which make the cohomology ring $\HG_{k,m}$ into a $\HG_k\otimes \HG_m$-module.
Since the algebra $\HG_m$ is commutative, we can turn a left $\HG_m$-module into
a right $\HG_m$-module.  Hence, we can make $\HG_{k,m}$ into a
$(\HG_k,\HG_m)$-bimodule.

Let $k_1,\ldots, k_m$ be a sequence of integers with $0 \leq k_i\leq N$ for all
$i$.  Form the $(\HG_{k_1},\HG_{k_m})$-bimodule
\[
 \HG_{k_1, \ldots, k_m} = \HG_{k_1,k_2} \otimes_{\HG_{k_2}} \HG_{k_2,k_3}
 \otimes_{\HG_{k_3}} \cdots \otimes_{\HG_{k_{m-1}}} \HG_{k_{m-1},k_m}~.
\]
Consider the partial flag variety $G_{k_1,\ldots,k_m}$ which consists of
sequences $(W_1,\ldots,W_m)$ of linear subspaces of $W$ such that the dimension
of $W_i$ is $k_i$ and $W_i \subset W_{i+1}$ if $k_i \leq k_{i+1}$ and $W_i
\supset W_{i+1}$ if $k_{i+1} \geq k_i$. The forgetful maps
\[
 \xymatrix{G_{k_1} & G_{k_1,\ldots,k_m} \ar[l]_{p_1} \ar[r]^{p_2} & G_{k_m}}
\]
induce maps of cohomology rings
\[
 \xymatrix{ \HG_{k_1} \ar[r]^-{p_1^*} & \HG(G_{k_1,\ldots,k_m},\Q)   & \HG_{k_m} \ar[l]_-{p_2^*}}
\]
which make the cohomology ring $\HG(G_{k_1,\ldots,k_m},\Q)$ into a graded
$(\HG_{k_1},\HG_{k_m})$-bimodule.  As one might expect, there is an isomorphism
\begin{equation} \label{eq_kh_iso}
H^*_{GL(N)}(G_{k_1,\ldots,k_m},\Q)\cong \HG_{k_1,\ldots,k_m}
\end{equation}
of graded $(\HG_{k_1},\HG_{k_m})$-bimodules.

%
\subsection{One step iterated flag varieties} \label{subsex_onestep}
%

A special role is played in our theory by the one step iterated flag varieties
\[
 G_{k,k+1} = \left\{ (W_k,W_{k+1}) |
 \dim_{\C} W_k = k, \; \dim_{\C} W_{k+1} =(k+1), \; 0
 \subset W_k \subset W_{k+1} \subset W  \right\}.
\]
Again, since $GL(N)$ acts transitively on $G_{k,k+1}$ we have
\begin{eqnarray}
  H_{GL(N)}^*(Gr(k,k+1.N)) = H_{{\rm Stab}(pt)}^*\left( pt \in G_{k,k+1}\right).
\end{eqnarray}
The stabilizer of a partial flag $(\C^k \subset \C^{k+1} \subset \C^N) \in
G_{k,k+1}$ is the group of invertible block $\big(k \times 1 \times (N-k-1)\big)$
upper--triangular matrices. This group is contractible onto its subgroup
$GL(k)\times GL(1) \times GL(N-k-1)$.  Therefore, the equivariant cohomology ring
of $G_{k,k+1}$ is
\begin{eqnarray}
 \HG_{k,k+1} &\cong& H_{GL(k)}(pt)\otimes H_{GL(1)}(pt) \otimes H_{GL(N-k-1)}(pt)
 \nn \\
  &\cong& \Q[w_{1},w_{2}, \ldots w_{k};\xi; z_{1},z_2,\ldots,
 z_{N-k-1}].
\end{eqnarray}
Although for our purposes it is useful to identify the generators $w_j$ and $z_j$
with their preimages under the inclusions:
\[
 \xymatrix{ \HG_k \ar[r]^-{p_1^*} & \HG_{k,k+1}   & \HG_{k+1} \ar[l]_-{p_2^*}} .
\]
These inclusions making $\HG_{k,k+1}$ an $(\HG_k,\HG_{k+1})$-bimodule are
explicitly given as follows:
\begin{eqnarray*}
 \HG_{k} & \xymatrix@1{\ar@{^{(}->}[r] & } & \HG_{k,k+1} \\
 x_{j,n} & \mapsto & w_j \qquad {\rm for}\; 1\leq
 j\leq k\\
 y_{1,n} & \mapsto & \xi+z_1 \\
 y_{\ell,n} & \mapsto & \xi \cdot z_{\ell-1}+z_{\ell} \qquad {\rm for}\; 1<
 \ell<N-k \\
 y_{N-k,n} & \mapsto & \xi\cdot z_{N-k-1}
\end{eqnarray*}
and
\begin{eqnarray*}
 \HG_{k+1} & \xymatrix@1{\ar@{^{(}->}[r] & }& \HG_{k,k+1}\\
 x_{1,n+2} & \mapsto & \xi+w_1 \\
 x_{j,n+2} & \mapsto & \xi \cdot w_{j-1}+x_j \qquad {\rm for}\; 1<
 j< k+1 \\
 x_{k+1,n+2} & \mapsto &  \xi \cdot w_{k} \\
 y_{\ell,n+2} & \mapsto & z_{\ell} \qquad \text{for $1 \leq \ell \leq N-k-1$}. \\
\end{eqnarray*}
Using these inclusions we identify certain generators of $\HG_k$ and $\HG_{k+1}$
with their images in the bimodule $\HG_{k,k+1}$, so that
\begin{eqnarray}
 \HG_{k,k+1} &\cong& \Q[x_{1,n},x_{2,n}, \ldots x_{k,n};\xi;y_{1,n+2},y_{2,n+2},\ldots,
 y_{N-k-1,n+2}].
\end{eqnarray}
We will refer to this set of generators as the canonical generators of the ring
$\HG_{k,k+1}$ in what follows.

The generators of $\HG_k$ and $\HG_{k+1}$ that are not mapped to canonical
generators can be expressed in terms of canonical generators as follows:
\begin{eqnarray}
x_{\alpha,n+2}  = x_{\alpha,n}+ x_{\alpha-1,n} \cdot \xi ~, \qquad \quad
y_{\alpha,n} & =& y_{\alpha,n+2}+ y_{\alpha-1,n+2} \cdot \xi ~.
\label{eq_canonical_gen}
\end{eqnarray}
Similarly, the canonical generators can be expressed using non-canonical
generators:
\begin{eqnarray}
x_{\alpha,n} = \sum_{\ell=0}^{\alpha}(-1)^{\ell} x_{\alpha-\ell,n+2} \cdot
\xi^{\ell}~, \qquad \quad y_{\alpha,n+2} =  \sum_{j=0}^{\alpha}(-1)^{j}
y_{\alpha-j,n} \cdot \xi^{j}~. \label{eq_noncanonical_gen}
\end{eqnarray}
In what follows we will often write equations as above, where we define
$x_{0,n}=1$, $y_{0,n}=1$ and
\begin{eqnarray}
 x_{j,n} = 0 &\qquad &\text{for $j$ not in the range $0\leq j \leq k$}, \\
 y_{\ell,n} = 0 &\qquad &\text{for $\ell$ not in the range $0\leq \ell \leq N-k$}.
\end{eqnarray}

%
\subsubsection{Inclusions of infinite Grassmannians}
%
We introduce special elements of $\HG_k$ obtained from the inclusions
$H_{GL(k)}(pt)=H^*(Gr(k,\infty)) \to \HG_{k}$ and
$H_{GL(N-k)}(pt)=H^*(Gr(N-k,\infty)) \to \HG_{k}$.

\begin{defn}
For $\alpha, \beta \in \Z$, let $X_{\alpha,n}=Y_{\beta,n}=0$ for $\alpha,\beta<0$
and inductively define the elements $X_{\alpha,n}$, $Y_{\beta,n} \in \HG_{k}$ by
setting $X_{0,n}=Y_{0,n}=1$, and
\begin{eqnarray}
 X_{\alpha,n} := - \sum_{j=1}^{\alpha} y_{j,n} X_{\alpha-j,n}, \qquad \quad
 Y_{\beta,n}  := - \sum_{j=1}^{\beta} x_{j,n} Y_{\beta-j,n}.
\end{eqnarray}
\end{defn}

For example,
\begin{eqnarray}
 Y_{0,n} = 1 , \quad
  Y_{1,n} = -x_{1,n} , \quad
  Y_{2,n} = x_{1,n}^2-x_{2,n} , \quad
  Y_{3,n} = -x_{3,2}+2x_{1,n}x_{2,n}-x_{1,n}^3.
\end{eqnarray}
For a given choice of $N$, with $n=2k-N$, we have by convention that $x_{j,n}=0$
for $j>k$. However, each element $Y_{\alpha,n}$, for $\alpha \in Z_+$, contains
some power of $x_{1,n}$ which for $N>0$ is never zero since there are no
relations on the generators of $\HG_{k}$.

The collection of elements $X_{j,n}$ and $Y_{j,n}$ also satisfy the equations
\begin{eqnarray}
   \left(
 1+x_{1,n} t + x_{2,n} t^2 +\cdots + x_{k,n} t^k
 \right)\left(1+Y_{1,n} t + Y_{2,n} t^2  + \cdots + Y_{j,n} t^j + \cdots \right)
  &=& 1 ,\nn\\
 \left(1+X_{1,n} t + X_{2,n} t^2  + \cdots + X_{j,n} t^j + \cdots \right)
 \left(
 1+y_{1,n} t + y_{2,n} t^2 +\cdots + y_{N-k,n} t^{N-k}
 \right) &=& 1 .\nn \\
\end{eqnarray}
This fact follows immediately from the definition since the homogeneous terms in
the above equations are given by:
\begin{eqnarray}
  \sum_{j=0}^{\beta} x_{j,n} Y_{\beta-j,n} = \delta_{\beta,0} ~,  \qquad \quad
  \sum_{j=0}^{\alpha} y_{j,n} X_{\alpha-j,n} = \delta_{\alpha,0}\label{eq_Y}\label{eq_X}
  .
\end{eqnarray}
The homogeneous terms of these expressions give the defining relations for the
cohomology rings $H^*(k,\infty)$ and $H^*(N-k,\infty)$, respectively. The ring $
H^*(k,\infty)$ is given by
\begin{eqnarray}
 H^*(k,\infty) &=& \Q[x_1,x_2,\cdots x_k; y_1, y_2, \cdots, y_{j}, \cdots]/I_{k,\infty}
\end{eqnarray}
where $I_{k,\infty}$ is the ideal generated by the homogeneous elements in
\begin{eqnarray}
\left(1+x_{1} t + x_{2} t^2  + \cdots + x_{k} t^k \right)
 \left(
 1+y_{1} t + y_{2} t^2 +\cdots + y_{j} t^{j} + \cdots
 \right) &=& 1.
\end{eqnarray}
Hence, we have an inclusion of $H^*(Gr(k,\infty))$ into the equivariant
cohomology ring $\HG_{k}$ given by
\begin{eqnarray}
H^*(Gr(k,\infty)) & \longrightarrow & \HG_{k} \\
  x_j & \mapsto & x_{j,n} \\
  y_{\ell} &\mapsto & Y_{\ell,n}.
\end{eqnarray}

Similarly, the cohomology ring $H^*(Gr(N-k,\infty))$, also generated by Chern
classes, is
\begin{eqnarray}
 H^*(k,\infty) &=& \Q[x_1,x_2,\cdots x_j, \cdots ; y_1, y_2, \cdots, y_{N-k}]/I_{N-k,\infty}
\end{eqnarray}
where $I_{N-k, \infty}$ is the ideal generated by the homogeneous terms in
\begin{eqnarray}
\left(1+x_{1} t + x_{2} t^2  + \cdots + x_{j} t^j + \cdots \right)
 \left(
 1+y_{1} t + y_{2} t^2 +\cdots + y_{N-k} t^{N-k}
 \right) &=& 1.
\end{eqnarray}
Thus, we have an inclusion
\begin{eqnarray}
 H^*(Gr(N-k,\infty)) & \longrightarrow & \HG_{k} \\
  x_j &\mapsto & X_{j,n} \\
  y_{\ell} &\mapsto& y_{j,n}.
\end{eqnarray}

Just as we have identified the canonical generators of $\HG_{k}$ and $\HG_{k+1}$
with generators or sums of generators in $\HG_{k,k+1}$, we can do the same with
the elements $X_{j,n}$ and $Y_{k,n}$.  Here we collect some identities involving
these elements in $\HG_{k,k+1}$.  The reader is encouraged to convert these
identities into a graphical form like the one used for the ordinary cohomology
rings in \cite{Lau3}.

\begin{prop} \label{prop_relation}For all $\alpha \in \Z_+$, the equations
\begin{eqnarray}
  X_{\alpha,n} &=& \sum_{\ell=0}^{\alpha} (-1)^{\ell} X_{\alpha-\ell,n+2} \cdot
  \xi^{\ell} \label{eq_Xslide}\\
  Y_{\alpha,n+2} &=& \sum_{\ell=0}^{\alpha} (-1)^{\ell} Y_{\alpha-\ell,n} \cdot
  \xi^{\ell} \label{eq_Yslide}\\
  \xi^{\alpha} &=& (-1)^{\alpha} \sum_{j=0}^{\alpha} x_{\alpha-j,n}Y_{j,n+2} \label{eq_dY} \\
  \xi^{\alpha} &=& (-1)^{\alpha} \sum_{j=0}^{\alpha} X_{\alpha-j,n}y_{j,n+2}
  \label{eq_dX}
\end{eqnarray}
hold in $\HG_{k,k+1}$.
\end{prop}

\begin{proof}
These relations follow from \eqref{eq_canonical_gen} and
\eqref{eq_noncanonical_gen} together with the definitions of $X_{\alpha,n}$ and
$Y_{\alpha,n}$.
\end{proof}

%
\subsection{Iterated flag varieties}
%

The $a$-step iterated flag variety $G_{k,k+1,\ldots,k+a}$ consists of
$a+1$-tuples
\[
  \left\{ (W_k,\ldots, W_{k+a}) |
 \dim_{\C} W_{k+j} = (k+j), \; 0 \subset W_k \subset W_{k+1} \subset \cdots
 \subset W_{k+a} \subset W\right\},
\]
where $0 \leq k \leq k+a \leq N$.  Using that $GL(N)$ acts transitively on this
variety, the cohomology ring of this variety is given by
\begin{equation}
\HG_{k,k+1,\ldots,k+a} \cong \Q[x_{1,n},\ldots, x_{k,n};\xi_{1},\xi_{2}
\ldots,\xi_{a}; y_{1,n+2a}, \ldots , y_{N-k-a,n+2a}] .
\end{equation}
However, this ring can also be equivalently described by
\begin{equation}
  \HG_{k,k+1,\ldots,k+a} \cong \HG_{k,k+1} \otimes_{\HG_{k+1}} H_{k+1,k+2}
  \otimes_{\HG_{k+2}} \cdots \otimes_{\HG_{k+a-1}} \HG_{k+a-1,k+a}.
\end{equation}

For example, the cohomology ring $\HG_{k,k+1,k+2} \cong \HG_{k,k+1}
\otimes_{\HG_{k+1}} H_{k+1,k+2}$.  The ring $\HG_{k,k+1,k+2}$ has generators
\begin{equation}
\HG_{k,k+1,k+2} \cong \Q[x_{1,n},\ldots, x_{k,n};\xi_{1},\xi_{2}; y_{1,n+4},
\ldots , y_{N-k-2,n+4}],
\end{equation}
but using the isomorphism $\HG_{k,k+1,k+2} \cong \HG_{k,k+1} \otimes_{\HG_{k+1}}
H_{k+1,k+2}$ we will write
\begin{eqnarray}
\HG_{k,k+1,k+2} &\longrightarrow & \HG_{k,k+1} \otimes_{\HG_{k+1}} H_{k+1,k+2} \\
 x_{j,n} & \mapsto& x_{j,n} \otimes 1 \nn \\
 y_{\ell,n+4} &\mapsto& 1 \otimes
y_{\ell,n+4} \nn \\
 \xi_1 &\mapsto & \xi \otimes 1 \nn \\
 \xi_2 & \mapsto& 1 \otimes \xi \nn
\end{eqnarray}
Notice that in the tensor product anything with the second label equal to $n+2$
can be moved from one tensor factor to the other, so that $x_{j,n}x_{s,n+2}
\otimes y_{\ell,n+4}=x_{j,n} \otimes x_{s,n+2}y_{\ell,n+4}$ and $x_{j,n}y_{s,n+2}
\otimes y_{\ell,n+4}=x_{j,n} \otimes y_{s,n+2}y_{\ell,n+4}$.

This convention extends to other iterated flag varieties as well.  The generators
of the equivariant cohomology ring $\HG_{k,k+1,k} \cong \HG_{k,k+1}
\otimes_{\HG_{k+1}} \HG_{k+1,k}$ of the variety $G_{k,k+1,k}$ can be written as
$x_{j,n}\otimes 1$, $1 \otimes x_{j,n}$, $\xi \otimes 1$, $1 \otimes \xi$, and
$y_{s,n+2} \otimes 1 = 1 \otimes y_{s,n+2}$.  However, the tensor product over
the action of $H_{k+1}$ induces some relations on these generators.  From
\eqref{eq_canonical_gen} we have that
\begin{eqnarray}
  x_{\alpha,n} \otimes 1 + x_{\alpha-1,n}\cdot \xi \otimes 1 \;= x_{\alpha,n+2}\otimes 1
  =\; 1 \otimes x_{\alpha,n+2} = 1 \otimes x_{\alpha,n} + 1\otimes x_{\alpha-1,n}\cdot
  \xi.
\end{eqnarray}
Relations among other elements of $\HG_{k,k+1} \otimes_{\HG_{k+1}} \HG_{k+1,k}$
can also be derived using \eqref{eq_noncanonical_gen}.

\begin{prop}\label{prop_two_def}
 In the ring $\HG_{k,k+1} \otimes_{\HG_{k+1}} \HG_{k+1,k}$ we have
\begin{eqnarray}
  \sum_{j=0}^{\alpha}(-1)^{j} x_{j,n} \otimes \xi^{\alpha-j} &=&
  \sum_{\ell=0}^{\alpha}(-1)^{\ell} \xi^{\alpha-\ell} \otimes x_{\ell,n}
\end{eqnarray}
for all $\alpha \in \Z_+$.  Similarly, in $\HG_{k,k-1} \otimes_{\HG_{k-1}}
\HG_{k-1,k}$ we have
\begin{eqnarray}
  \sum_{j=0}^{\alpha}(-1)^{j} y_{j,n} \otimes \xi^{\alpha-j} &=&
  \sum_{\ell=0}^{\alpha}(-1)^{\ell} \xi^{\alpha-\ell} \otimes y_{\ell,n}
\end{eqnarray}
for all $\alpha \in \Z^+$.
\end{prop}

\begin{proof} The first identity is proven as follows:
\begin{eqnarray}
  \sum_{j=0}^{\alpha}(-1)^{j} x_{j,n} \otimes \xi^{\alpha-j}
  &\refequal{\eqref{eq_dY}}&
  \sum_{j=0}^{\alpha}\sum_{\ell=0}^{\alpha-j}
  (-1)^{j} x_{j,n} \otimes (-1)^{\alpha-j}
   Y_{\alpha-j-\ell,n+2}x_{\ell,n} \\
   &=& \sum_{j=0}^{\alpha}\sum_{\ell=0}^{\alpha-j}
  (-1)^{\alpha} x_{j,n}Y_{\alpha-j-\ell,n+2} \otimes
   x_{\ell,n} \\
   &=&
   \sum_{\ell=0}^{\alpha}
  (-1)^{\alpha} \sum_{j=0}^{\alpha-\ell}x_{j,n}Y_{(\alpha-\ell)-j,n+2} \otimes
   x_{\ell,n}.
\end{eqnarray}
Using \eqref{eq_dY} once more completes the proof.  The second identity is proven
similarly.
\end{proof}

\begin{prop}[Dot slide] \label{prop_dot_slide}The equations
 \begin{eqnarray}
   \sum_{\ell=0}^{k}(-1)^{\ell} \xi^{k-\ell+1} \otimes x_{\ell,n} &=&
   \sum_{\ell=0}^{k}(-1)^{\ell} \xi^{k-\ell} \otimes x_{\ell,n} \cdot \xi\\
   \sum_{j=0}^{N-k}(-1)^{j} y_{j,n}\cdot \xi \otimes \xi^{k-j}
   &=&
   \sum_{j=0}^{N-k}(-1)^{j} y_{j,n} \otimes \xi^{k-j+1}
 \end{eqnarray}
hold in the rings $\HG_{k,k+1} \otimes_{\HG_{k+1}} \HG_{k+1,k}$ and $\HG_{k,k-1}
\otimes_{\HG_{k-1}} \HG_{k-1,k}$, respectively.
\end{prop}

This Proposition says that multiplication by $\xi$ on one of the tensor factors
in the above sums is equivalent to multiplication by $\xi$ on the other tensor
factor. We will see in the next section that this proposition is used to show
that a dot can be slid from one side of a cup to the other
(Lemma~\ref{lem_dot_slide}).

\begin{proof}  We have
\begin{eqnarray}
 \sum_{\ell=0}^{k}(-1)^{\ell} \xi^{k-\ell+1} \otimes x_{\ell,n}
  &\refequal{\eqref{eq_dY}}&
 \sum_{\ell=0}^{k}\sum_{j=0}^{k+1 -\ell}(-1)^{k+1} \cdot x_{j,n}Y_{k+1-\ell-j,n+2}
  \otimes x_{\ell,n} \\
  &=&
 \sum_{j=0}^{k+1}(-1)^{k+1}  x_{j,n}
  \otimes \sum_{\ell=0}^{k+1-j}Y_{(k+1-j)-\ell,n+2}x_{\ell,n} \\
  &\refequal{\eqref{eq_dY}}&
  \sum_{j=0}^{k+1}(-1)^{j}  x_{j,n}
  \otimes \xi^{k-\ell+1}
\end{eqnarray}
were in the second equality we have used that $x_{k+1,n}=0$ to increase the
$\ell$-summation index and then switched the summation order.  The first claim
now follows from \eqref{prop_two_def}.  The second claim is proven similarly.
\end{proof}

\begin{cor} \label{cor_bimodule_maps}
The assignments
\[
\begin{array}{ccl}
    \HG_k \ & \longrightarrow &
    \left(\HG_{k,k+1} \otimes_{\HG_{k+1}} \HG_{k+1,k}\right)
    \\
     1 &\mapsto&
     \xsum{j=0}{k}(-1)^{j} x_{j,n} \otimes \xi^{k-j} =
     \xsum{\ell=0}{k}(-1)^{\ell} \xi^{k-\ell} \otimes x_{\ell,n}
 \end{array}
 \]
 and
\[
 \begin{array}{ccl}
    \HG_k \ & \longrightarrow &
    \left(\HG_{k,k-1} \otimes_{\HG_{k-1}} \HG_{k-1,k}\right)
    \\
     1 &\mapsto&
    \xsum{j=0}{N-k}(-1)^{j} y_{j,n} \otimes \xi^{N-k-j} =
     \xsum{\ell=0}{N-k}(-1)^{\ell} \xi^{N-k-\ell} \otimes y_{\ell,n}
 \end{array}
 \]
define morphisms of graded bimodules of degree $2k$ and $2(N-k)$, respectively.
\end{cor}

\begin{proof}
We prove the first claim, the second is left to the reader. It suffices to check
that the left action of $\HG_k$ on the image of $1 \in \HG_k$ in
$\left(\HG_{k,k+1} \otimes_{\HG_{k+1}} \HG_{k+1,k}\right)$ is equal to right
action of $\HG_k$.  The Corollary follows since $x_{j,n}$ can be moved across
tensor products by \eqref{eq_canonical_gen} and \eqref{eq_noncanonical_gen} at
the cost of introducing powers of $\xi$ on one of the tensor factors. By
Proposition~\ref{prop_dot_slide}, factors of $\xi$ can be slid across tensor
factors in the above sums.
\end{proof}

%
\subsection{Defining the 2-category $\GrGs$}
%

Recall the additive 2-category $\cat{Bim}$ whose objects are graded rings,
morphisms are graded bimodules, and the 2-morphisms are degree-preserving
bimodule maps. Idempotent bimodule homomorphisms split in $\cat{Bim}$. Denote by
$\cat{Bim}^*$ the graded additive 2-category with the same objects and
1-morphisms as $\cat{Bim}$, and whose 2-morphisms consist of all bimodule maps,
each of which is the sum of its homogeneous components.

Fix a positive integer $N$ and let $n=2k-N$ for $0 \leq k \leq N$.

\begin{defn}
The 2-category $\GrG$ is the idempotent completion inside of $\cat{Bim}$ (see
\cite[Section~6.1.5]{Lau3}) of the 2-category consisting of
\begin{itemize}
  \item objects: the graded equivariant cohomology rings $\HG_k$ for each $0 \leq k \leq N$,
  \item morphisms: generated by the graded ($\HG_k$,$\HG_k$)-bimodule $\HG_k$, the
  graded ($\HG_k$,$\HG_{k+1}$)-bimodule $\HG_{k,k+1}$ and the graded ($\HG_{k+1}$,$\HG_{k}$)-bimodule
  $\HG_{k+1,k}$, together with their shifts $\HG_k\{s\}$, $\HG_{k,k+1}\{s\}$, and
  $\HG_{k,k+1}\{s\}$ for $s \in \Z$.  The bimodules $\HG_k=\HG_k\{0\}$ are the
  identity 1-morphisms.  Thus, a generic morphism from $\HG_{k_1}$ to $\HG_{k_m}$ is a
  direct sum of graded $(\HG_{k_1},\HG_{k_m})$-bimodules of the form
\begin{equation} \label{eq_long_bimodule}
   \HG_{k_1,k_2} \otimes_{\HG_{k_2}} \otimes_{\HG_{k_2,k_3}} \otimes_{\HG_{k_3}}
   \cdots \otimes_{\HG_{k_{m-1}}} \HG_{k_{m-1},k_m} \{s\}
\end{equation}
 where $k_{i+1} = k_{i} \pm 1$ for $1<i\leq m$.
  \item 2-morphisms: degree-preserving bimodule maps
\end{itemize}
\end{defn}
There is a 2-subcategory $\GrGs$ of $\cat{Bim}^*$ with the same objects and
morphisms as $\GrG$, and with 2-morphisms
\begin{equation}
  \GrGs(x,y) := \bigoplus_{s \in \Z} \GrG(x\{s\},y).
\end{equation}

Given two $(\HG_{k_1},\HG_{k_m})$-bimodules $M_1$ and $M_2$ of the form
\eqref{eq_long_bimodule} there are canonical inclusions of $\HG_{k_1}$ into the
first factor of the tensor product and canonical inclusion of $\HG_{k_m}$ into
the last factor.  Let $n_1=2k_1-N$ and $n_m=2k_m-N$. Under these inclusions, the
left action of elements $x_{j,n_1}, y_{\ell,n_1} \in \HG_{k_1}$ can be written as
\[
x_{j,n_1} \otimes 1 \otimes \dots \otimes 1, \quad y_{\ell,n_1} \otimes 1 \otimes
\dots \otimes 1
\]
in either bimodule $M_1$ or $M_2$.  Likewise, the right action of $x_{j,n_m},
y_{\ell,n_m} \in \HG_{k_m}$ can be written as
\[
1 \otimes  \dots \otimes  1 \otimes x_{j,n_m}, \quad 1 \otimes  \dots \otimes 1
\otimes y_{\ell,n_m},
\]
in $M_1$ or $M_2$. As in \cite{Lau3}, to define an
$(\HG_{k_1},\HG_{k_m})$-bimodule homomorphism from $M_1$ to $M_2$ we specify the
map on elements of the form $\xi^{\alpha_1} \otimes \xi^{\alpha_2} \otimes \dots
\otimes \xi^{\alpha_m}$ with it understood that preservation of left action of
$H_{k_1}$, and the right action of $H_{k_m}$, require that elements
\[
 x_{j,n_1} \otimes 1 \otimes \dots \otimes 1, \quad y_{\ell,n_1} \otimes 1 \otimes \dots \otimes
 1, \quad
 1 \otimes  \dots \otimes  1 \otimes x_{j,n_m}, \quad 1 \otimes  \dots \otimes 1 \otimes
 y_{\ell,n_m},
\]
of $M_1$ are mapped to the corresponding elements of $M_2$.

%
\section{Equivariant representations of $\UcatD$}
%

%
\subsection{Defining the 2-functor $\Gamma_{N}^G$} \label{subsec_define_gamma}
%

On objects the 2-functor $\Gamma_{N}^{G}\maps \Ucatq \to \GrGs$ sends $\bfit{n}$
to the ring $\HG_k$ whenever $n$ and $k$ are compatible:
\begin{eqnarray}
 \Gamma_{N}^{G} \maps \Ucatq & \to & \GrGs \nn \\
  \bfit{n} & \mapsto &
  \left\{\begin{array}{ccl}
    \HG_k & & \text{with $n=2k-N$ and $0\leq k \leq N$} \\
    0  & & \text{otherwise.}
  \end{array} \right.
\end{eqnarray}
Morphisms of $\Ucatq$ get mapped by $\Gamma^{G}_N$ to graded bimodules:
\begin{eqnarray}
 \Gamma_{N}^{G} \maps \Ucatq & \to & \GrGs \nn \\
  \onen\{s\} & \mapsto &
  \left\{\begin{array}{ccl}
    \HG_k\{s\} & & \text{with $n=2k-N$ and $0\leq k \leq N$} \\
    0  & & \text{otherwise.}
  \end{array} \right. \\
  \cal{E}\onen\{s\} & \mapsto &
  \left\{\begin{array}{ccl}
    \HG_{k,k+1}\{s+1-N+k\} & & \text{with $n=2k-N$ and $0\leq k < N$} \\
    0  & & \text{otherwise.}
  \end{array} \right. \\
  \cal{F}\onen\{s\} & \mapsto &
  \left\{\begin{array}{ccl}
    \HG_{k+1,k}\{s+1-k\} & & \text{with $n=2k-N$ and $0\leq k < N$} \\
    0  & & \text{otherwise.}
  \end{array} \right.
\end{eqnarray}

Here $\HG_{k,k+1}\{s+1-k\}$ is the bimodule $\HG_{k,k+1}$ with the grading
shifted by $s+1-k$ so that $$(\HG_{k,k+1}\{s+1-k\})_j =
(\HG_{k,k+1})_{j+s+1-k}.$$ More generally, we have
\begin{eqnarray}
  \cal{E}^{\alpha}\onen\{s\} &\mapsto &
  \HG_{k,k+1,k+2,\cdots,k+(\alpha-1),k+\alpha}\{s+\alpha(-N+k) +\alpha\} \nn\\
   \cal{F}^{\beta}\onen\{s\} &\mapsto &
    \HG_{k,k-1,k-2,\cdots,k-(\beta-1),k-\beta}\{s-\beta k+2-\beta\} \nn
\end{eqnarray}
so that \begin{equation} \label{eq_long_composite}\cal{E}^{\alpha_1}
\cal{F}^{\beta_1}\cal{E}^{\alpha_2} \cdots
 \cal{F}^{\beta_{k-1}}\cal{E}^{\alpha_k}\cal{F}^{\beta_k}\onen\{s\}
 \cong \cal{E}^{\alpha_1}\mathbf{1}_{n-\sum (\beta_j-\alpha_j)} \circ \dots \cal{E}^{\alpha_k}\mathbf{1}_{n-2\beta_k} \circ \cal{F}^{\beta_k}\onen\{s\}
\end{equation}
is mapped to the graded bimodule
\begin{eqnarray}
\HG_{k,k-1,\cdots,k-\beta_k,k-\beta_k+1,\cdots,k-\beta_k+\alpha_k,
k-\beta_k+\alpha_k+1, \cdots, k-\sum_j(\beta_j+\alpha_j)}\nn
\end{eqnarray}
with grading shift $\{s+s'\}$, where $s'$ is the sum of the grading shifts for
each terms of the composition in \eqref{eq_long_composite}.  Formal direct sums
of morphisms of the above form are mapped to direct sums of the corresponding
bimodules.

It follows from \eqref{eq_kh_iso} that $\Gamma_{N}^{G}$ preserves composites up
to isomorphism. Hence, the 2-functor $\Gamma_N^{G}$ is a weak 2-functor or
bifunctor. In what follows we will often simplify our notation and write
$\Gamma^{G}$ instead of $\Gamma_{N}^{G}$. We now proceed to define $\Gamma^{G}$
on 2-morphisms.

%
\subsubsection{Biadjointness}
%

\begin{defn} \label{def_biadjoint}
The 2-morphisms generating biadjointness in $\Ucatq$ are mapped by $\Gamma^G$ to
the following bimodule maps.
\begin{eqnarray}
    \Gamma^G\left(\;\xy
    (0,-3)*{\bbpef{}};
    (8,-5)*{ \bfit{n}};
    (-4,2)*{\scs \cal{F}};
    (4,2)*{\scs \cal{E}};
    \endxy \; \right)
    & : &
 \left\{
  \begin{array}{ccl}
    \HG_k \ & \longrightarrow &
    \left(\HG_{k,k+1} \otimes_{\HG_{k+1}} \HG_{k+1,k}\right)\{1-N\}
    \\
     1 &\mapsto&
     \xsum{j=0}{k}(-1)^{j} x_{j,n} \otimes \xi^{k-j} =
     \xsum{\ell=0}{k}(-1)^{\ell}  \xi^{k-\ell} \otimes x_{\ell,n}
 \end{array}
 \right.
     \label{def_eq_FE_G}\\
    \Gamma^G\left(\;\xy
    (0,-3)*{\bbpfe{}};
    (8,-5)*{ \bfit{n}};
    (-4,2)*{\scs \cal{E}};
    (4,2)*{\scs \cal{F}};
    \endxy\;\right)
     & : &
     \left\{
  \begin{array}{ccl}
    \HG_k \ & \longrightarrow &
    \left(\HG_{k,k-1} \otimes_{\HG_{k-1}} \HG_{k-1,k}\right)\{1-N\}
    \\
     1 &\mapsto&
    \xsum{j=0}{N-k}(-1)^{j} y_{j,n} \otimes \xi^{N-k-j} =
     \xsum{\ell=0}{N-k} (-1)^{\ell} \xi^{N-k-\ell} \otimes y_{\ell,n}
 \end{array}
 \right.
     \label{def_eq_EF_G}
\end{eqnarray}
\begin{equation}
   \Gamma^G\left(\;\xy
    (0,0)*{\bbcef{}};
    (8,2)*{ \bfit{n}};
    (-4,-5.5)*{\scs \cal{F}};
    (4,-5.5)*{\scs \cal{E}};
    \endxy\;\right)
 :
 \left\{
  \begin{array}{ccl}
     \left(\HG_{k,k+1} \otimes_{\HG_{k+1}} \HG_{k+1,k}\right)\{1-N\}
     & \rightarrow &
     \HG_k
    \\ \\
    \xi_1^{m_1} \otimes \xi_1^{m_2}
    &\mapsto &
    (-1)^{m_1+m_2+k-N+1}\; X_{m_1+m_2+1+k-N,n}
 \end{array}
 \right.
   \label{def_FE_cap} \end{equation}\begin{equation}
  \Gamma^G\left(\;\xy
    (0,0)*{\bbcfe{}};
    (8,2)*{ \bfit{n}};
    (-4,-5.5)*{\scs \cal{E}};
    (4,-5.5)*{\scs \cal{F}};
    \endxy \right)
 :
 \left\{
  \begin{array}{ccl}
     \left(\HG_{k,k-1} \otimes_{\HG_{k-1}} \HG_{k-1,k}\right)\{1-N\}
     & \rightarrow &
     \HG_k
    \\ \\
    \xi_1^{m_1} \otimes \xi_1^{m_2}
    &\mapsto &
    (-1)^{m_1+m_2+1-k}\;Y_{m_1+m_2+1-k,n }
 \end{array}
 \right.
\label{def_EF_cap}
\end{equation}
\end{defn}
Corollary~\ref{cor_bimodule_maps} shows that the cups above are well-defined
bimodule maps. It is clear that the caps are bimodule maps since the image of
$\xi_1^{m_1} \otimes \xi_2^{m_2}$ only depends on the sum $m_1+m_2$.

These definitions preserve the degree of the 2-morphisms of $\Ucatq$ defined in
Section~\ref{subsec_Ucat}.  In \eqref{def_eq_FE_G} the element $1$ is in degree
zero and is mapped to a sum of elements in degree $2k$ that have been shifted by
$\{1-N\}$ for a total degree $2k+1-N= 1 +n$.  The degree in \eqref{def_eq_EF_G}
is $2(N-k)+(1-N)=1-n$.  Similarly, in \eqref{def_FE_cap} a degree $2(m_1+m_2)$
element shifted by $1-N$ is mapped to a degree $2(m_1+m_2+k-N-1)$ element, so the
bimodule map has total degree $1+n$.  One can easily check that the bimodule map
defined in \eqref{def_EF_cap} is of degree $1-n$.

Notice the similarity between the definitions
\eqref{def_eq_FE_G}--\eqref{def_EF_cap} defining the 2-functor $\Gamma_N^G$ and
the corresponding definitions for the 2-functor $\Gamma_N$ defined in
\cite{Lau3}.  In fact, in $\Grs$ the definitions agree.  That is, in the presence
of the extra relations imposed on the ordinary cohomology rings, the above
assignments for $\Gamma_N^G$ are equal to the assignments made by $\Gamma_N$.

%
\subsubsection{NilHecke generators}
%

We show that the nilHecke algebra $\BNC_a$ acts on $\End(\HG_{k,k+1,\ldots,
k+a})$ with $\chi_i$ acting by multiplication by $\xi_i$ and $u_i$ acting by the
divided differences $\partial_i$ operator on the variables $\xi_i$.

\begin{defn}
The 2-morphisms $z_n$ and $\hat{z}_n$ in $\Ucatq$ are mapped by $\Gamma_N^G$ to
the graded bimodule maps:
\begin{eqnarray}
  \Gamma^G\left(\xy
 (0,8);(0,-8); **\dir{-} ?(.75)*\dir{>}+(2.3,0)*{\scriptstyle{}};
 (0,0)*{\txt\large{$\bullet$}};
 (4,-3)*{ \bfit{n}};
 (-6,-3)*{ \bfit{n+2}};
 (-10,0)*{};(10,0)*{};
 \endxy\right)
\quad &\maps&
 \left\{
\begin{array}{ccc}
  \HG_{k,k+1}\{1-N+k\} & \to  & \HG_{k,k+1}\{1-N+k\} \\ \\
  \xi^m & \mapsto &  \quad \xi^{m+1}
\end{array}
\right. \\
  \Gamma^G\left(\xy
 (0,8);(0,-8); **\dir{-} ?(.75)*\dir{<}+(2.3,0)*{\scriptstyle{}};
 (0,0)*{\txt\large{$\bullet$}};
 (6,-3)*{ \bfit{n+2}};
 (-4,-3)*{ \bfit{n}};
 (-10,0)*{};(10,0)*{};
 \endxy\right)
\quad &\maps&
 \left\{
\begin{array}{ccc}
  \HG_{k+1,k}\{1-k\} & \to  & \HG_{k+1,k}\{1-k\} \\ \\
  \xi^m & \mapsto &  \quad \xi^{m+1}
\end{array}
\right.
\end{eqnarray}
Note that these assignment are degree preserving since these bimodule maps are
degree 2.
\end{defn}

The nilCoxeter generator $U_n$ is mapped to the bimodule map which acts as the
divided difference operator in the variables $\xi_j$ for $1 \leq j \leq a$.
\begin{defn} \label{def_nil_generator}
Define bimodule maps by
\begin{eqnarray}
  \Gamma^G\left(   \xy
    (0,0)*{\twoIu};
    (6,0)*{\bfit{n}};
 \endxy\right)
 & : &
 \left\{
 \begin{array}{ccl}
   \HG_{k,k+1,k+2}\{1-N\} & \to & \HG_{k,k+1,k+2}\{1-N\} \\ & &\\
 \xi_1^{m_1} \otimes \xi_2^{m_2}
  & \mapsto &
\xsum{j=0}{m_1-1} \xi_1^{m_1+m_2-1-j} \otimes \xi_2^{j}-
        \xsum{j=0}{m_2-1}
        \xi_1^{m_1+m_2-1-j} \otimes \xi_2^{j}
 \end{array}
    \right.
 \nn \\
  \Gamma^G\left(   \xy
    (0,0)*{\twoId};
    (6,0)*{\bfit{n+4}};
 \endxy\right)
 & : &
 \left\{
 \begin{array}{ccl}
   \HG_{k+2,k+1,k}\{1-N\} & \to & \HG_{k+2,k+1,k}\{1-N\} \\ & & \\
 \xi_1^{m_1} \otimes \xi_2^{m_2}
  & \mapsto &
\xsum{j=0}{m_2-1} \xi_1^{m_1+m_2-1-j} \otimes \xi_2^{j}
           -
        \xsum{j=0}{m_1-1}
        \xi_1^{m_1+m_2-1-j} \otimes \xi_2^{j}
 \end{array}
  \right.
 \nn \\
 & & \label{eq_def_gamma_Un}
\end{eqnarray}
The value of these maps on any other generator is determined from the rules above
together with the requirement that the maps preserve the actions of $\HG_k$ and
$\HG_{k+2}$. The maps $\Gamma^{G}(U_n)$ and $\Gamma^{G}(\hat{U}_n)$, as defined
above, have degree $-2$ so that $\Gamma^{G}$ preserves the degree of $U_n$ and
$\hat{U}_n$.
\end{defn}

\begin{rem} \label{rem_nil_zero}
$\Gamma^{G}(U_n)(\xi_1^{m_1}\otimes \xi_2^{m_2})$ is zero when $m_1=m_2$.  This
is clear from the definition.  When the number of dots the upward oriented lines
is equal the two sums above cancel.
\end{rem}

%
\subsection{Checking the relations of $\UcatD$}
%

In this section we show that the relations of Section~\ref{subsec_Ucat} for
$\Ucatq$ are satisfied in $\GrGs$, thus establishing that $\Gamma_N^{G}$ is a
2-functor.  The proof is analogous to the proof given in \cite{Lau3} that
$\Gamma_N \maps \Ucatq \to \Grs$ is a 2-functor .  Here we must be careful that
we never use the additional relations from the ordinary cohomology rings.  From
the definitions in the previous section it is clear that $\Gamma_N^{G}$ preserves
the degree associated to generators. For this reason, we often simplify our
notation in this section by omitting the grading shifts when no confusion is
likely to arise.

\begin{lem}[Biadjointness] \label{lem_biadjoint}
In $\GrGs$ the following identities are satisfied
\[
  \Gamma^{G}\left(\;  \xy   0;/r.18pc/:
    (-8,0)*{}="1";
    (0,0)*{}="2";
    (8,0)*{}="3";
    (-8,-10);"1" **\dir{-};
    "1";"2" **\crv{(-8,8) & (0,8)} ?(0)*\dir{>} ?(1)*\dir{>};
    "2";"3" **\crv{(0,-8) & (8,-8)}?(1)*\dir{>};
    "3"; (8,10) **\dir{-};
    (14,10)*{ \bfit{n}};
    (-6,10)*{ \bfit{n+2}};
    \endxy \;\right)
    \; =
    \;
       \Gamma^{G}\left(\;\xy   0;/r.18pc/:
    (-8,0)*{}="1";
    (0,0)*{}="2";
    (8,0)*{}="3";
    (0,-10);(0,10)**\dir{-} ?(.5)*\dir{>};
    (5,10)*{ \bfit{n}};
    (-8,10)*{ \bfit{n+2}};
    \endxy\;\right)
\qquad
    \Gamma^{G}\left(\;\xy   0;/r.18pc/:
    (-8,0)*{}="1";
    (0,0)*{}="2";
    (8,0)*{}="3";
    (-8,-10);"1" **\dir{-};
    "1";"2" **\crv{(-8,8) & (0,8)} ?(0)*\dir{<} ?(1)*\dir{<};
    "2";"3" **\crv{(0,-8) & (8,-8)}?(1)*\dir{<};
    "3"; (8,10) **\dir{-};
    (12,10)*{ \bfit{n}};
    (-6,10)*{ \bfit{n-2}};
    \endxy\;\right)
    \; =
    \;
       \Gamma^{G}\left(\;\xy   0;/r.18pc/:
    (-8,0)*{}="1";
    (0,0)*{}="2";
    (8,0)*{}="3";
    (0,-10);(0,10)**\dir{-} ?(.5)*\dir{<};
    (6,10)*{ \bfit{n}};
    (-7,10)*{ \bfit{n-2}};
    \endxy\;\right)
\]
\[
  \Gamma^{G}\left(\;  \xy   0;/r.18pc/:
    (8,0)*{}="1";
    (0,0)*{}="2";
    (-8,0)*{}="3";
    (8,-10);"1" **\dir{-};
    "1";"2" **\crv{(8,8) & (0,8)} ?(0)*\dir{>} ?(1)*\dir{>};
    "2";"3" **\crv{(0,-8) & (-8,-8)}?(1)*\dir{>};
    "3"; (-8,10) **\dir{-};
    (14,-10)*{ \bfit{n}};
    (-5,-10)*{ \bfit{n+2}};
    \endxy \;\right)
    \; =
    \;
       \Gamma^{G}\left(\;\xy 0;/r.18pc/:
    (8,0)*{}="1";
    (0,0)*{}="2";
    (-8,0)*{}="3";
    (0,-10);(0,10)**\dir{-} ?(.5)*\dir{>};
    (5,-10)*{ \bfit{n}};
    (-8,-10)*{ \bfit{n+2}};
    \endxy\;\right)
\qquad
    \Gamma^{G}\left(\;\xy  0;/r.18pc/:
    (8,0)*{}="1";
    (0,0)*{}="2";
    (-8,0)*{}="3";
    (8,-10);"1" **\dir{-};
    "1";"2" **\crv{(8,8) & (0,8)} ?(0)*\dir{<} ?(1)*\dir{<};
    "2";"3" **\crv{(0,-8) & (-8,-8)}?(1)*\dir{<};
    "3"; (-8,10) **\dir{-};
    (12,-10)*{ \bfit{n}};
    (-6,-10)*{ \bfit{n-2}};
    \endxy\;\right)
    \; =
    \;
       \Gamma^{G}\left(\;\xy  0;/r.18pc/:
    (8,0)*{}="1";
    (0,0)*{}="2";
    (-8,0)*{}="3";
    (0,-10);(0,10)**\dir{-} ?(.5)*\dir{<};
    (6,-10)*{ \bfit{n}};
    (-7,-10)*{ \bfit{n-2}};
    \endxy\;\right)
\]
for all $n \in \Z$.
\end{lem}

\begin{proof}
Consider the bimodule map
\begin{equation}
\Gamma^{G}\left(\;  \xy 0;/r.16pc/:
    (-8,0)*{}="1";
    (0,0)*{}="2";
    (8,0)*{}="3";
    (-8,-10);"1" **\dir{-};
    "1";"2" **\crv{(-8,8) & (0,8)} ?(0)*\dir{>} ?(1)*\dir{>};
    "2";"3" **\crv{(0,-8) & (8,-8)}?(1)*\dir{>};
    "3"; (8,10) **\dir{-};
    (14,10)*{ \bfit{n}};
    (-6,10)*{ \bfit{n+2}};
    \endxy \;\right) \maps \xi^{\alpha} \to \sum_{j=0}^{k}(-1)^{j}(-1)^{\alpha+k-j-k}
    Y_{\alpha-j,n+2} \cdot x_{j,n}
\end{equation}
After simplifying, the image of $\xi^{\alpha}$ is
\begin{eqnarray}
    \sum_{j=0}^{k}(-1)^{\alpha}
    Y_{\alpha-j,n+2} \cdot x_{j,n} \quad \refequal{\eqref{eq_Y}} \quad \xi^{\alpha}
\end{eqnarray}
which, depending on whether $\alpha \leq k$ or $\alpha>k$, uses that
$Y_{\ell,n+2}=0$ for $\ell<0$, or $x_{\ell,n}=0$ for $\ell>k$.  Since
$\xi^{\alpha}$ is fixed by $\Gamma^{G}\left(\;\xy   0;/r.14pc/:
    (-8,0)*{}="1";
    (0,0)*{}="2";
    (8,0)*{}="3";
    (0,-10);(0,10)**\dir{-} ?(.5)*\dir{>};
    (5,10)*{ \bfit{n}};
    (-8,10)*{ \bfit{n+2}};
    \endxy\;\right)$ we have established the first identity.  The others are
    proven similarly using \eqref{eq_Y}.
\end{proof}

\begin{lem}[Duality for $z_n$] \label{lem_dot_slide}
The equations
\begin{equation} \label{prop_eq_zdual}
 \Gamma^{G}\left(    \xy 0;/r.18pc/:
    (-8,5)*{}="1";
    (0,5)*{}="2";
    (0,-5)*{}="2'";
    (8,-5)*{}="3";
    (-8,-10);"1" **\dir{-};
    "2";"2'" **\dir{-} ?(.5)*\dir{<};
    "1";"2" **\crv{(-8,12) & (0,12)} ?(0)*\dir{<};
    "2'";"3" **\crv{(0,-12) & (8,-12)}?(1)*\dir{<};
    "3"; (8,10) **\dir{-};
    (15,-9)*{ \bfit{n+2}};
    (-12,9)*{ \bfit{n}};
    (0,4)*{\txt\large{$\bullet$}};
    \endxy \right)
    \quad =
    \quad
      \Gamma^{G}\left( \xy 0;/r.18pc/:
    (-8,0)*{}="1";
    (0,0)*{}="2";
    (8,0)*{}="3";
    (0,-10);(0,10)**\dir{-} ?(.5)*\dir{<};
    (10,5)*{ \bfit{n+2}};
    (-8,5)*{ \bfit{n}};
    (0,4)*{\txt\large{$\bullet$}};
    \endxy \right)
    \quad =
    \quad
    \Gamma^{G}\left( \xy 0;/r.18pc/:
    (8,5)*{}="1";
    (0,5)*{}="2";
    (0,-5)*{}="2'";
    (-8,-5)*{}="3";
    (8,-10);"1" **\dir{-};
    "2";"2'" **\dir{-} ?(.5)*\dir{<};
    "1";"2" **\crv{(8,12) & (0,12)} ?(0)*\dir{<};
    "2'";"3" **\crv{(0,-12) & (-8,-12)}?(1)*\dir{<};
    "3"; (-8,10) **\dir{-};
    (15,-9)*{ \bfit{n+2}};
    (-12,9)*{ \bfit{n}};
    (0,4)*{\txt\large{$\bullet$}};
    \endxy \right)
\end{equation}
of bimodule maps hold in $\GrGs$ for all $n \in \Z$.
\end{lem}

\begin{proof}
Since we have already established that $\Gamma^{G}$ preserves the biadjoint
structure of $\Ucatq$ in Lemma~\ref{lem_biadjoint}, the above
\eqref{prop_eq_zdual} is equivalent to proving
\begin{eqnarray}
    \text{ $\Gamma^G\left(\vcenter{\xy 0;/r.18pc/:
      (-3,8)*{};(0,5)*{}="2";(0,-5)*{}="2'"; (8,-5)*{}="3";
    "2";"2'" **\dir{-};
    "2'";"3" **\crv{(0,-12) & (8,-12)}?(0)*\dir{>}?(.97)*\dir{>};
    "3"; (8,5) **\dir{-};
    (15,-9)*{ \bfit{n}};
    (0,0)*{\txt\large{$\bullet$}};
    \endxy}\right)$}
   \;  =  \;\;
     \text{$ \Gamma^G\left(\vcenter{\xy 0;/r.18pc/:
    (0,5)*{}="2";(-3,8)*{}; (0,-5)*{}="2'"; (8,-5)*{}="3";
    "2";"2'" **\dir{-};
    "2'";"3" **\crv{(0,-12) & (8,-12)}?(0)*\dir{>}?(.97)*\dir{>};
    "3"; (8,5) **\dir{-};
    (15,-9)*{ \bfit{n}};
    (8,0)*{\txt\large{$\bullet$}};
    \endxy}\right)$} \qquad \quad
       \text{ $\Gamma^G\left(\vcenter{\xy 0;/r.18pc/:
    (0,5)*{}="2"; (0,-5)*{}="2'"; (8,-5)*{}="3"; (-3,8)*{};
    "2";"2'" **\dir{-};
    "2'";"3" **\crv{(0,-12) & (8,-12)}?(0)*\dir{<}?(.97)*\dir{<};
    "3"; (8,5) **\dir{-};
    (14,-9)*{ \bfit{n}};
    (0,0)*{\txt\large{$\bullet$}};
    \endxy}\right)$}
   \;   =  \;\;
     \text{$\Gamma^G\left(\vcenter{ \xy 0;/r.18pc/:
    (0,5)*{}="2"; (0,-5)*{}="2'"; (8,-5)*{}="3"; (-3,8)*{};
    "2";"2'" **\dir{-};
    "2'";"3" **\crv{(0,-12) & (8,-12)}?(0)*\dir{<}?(.97)*\dir{<};
    "3"; (8,5) **\dir{-};
    (15,-9)*{ \bfit{n}};
    (8,0)*{\txt\large{$\bullet$}};
    \endxy}\right)$} \nn
\end{eqnarray}
for all $n \in \Z$. Up to a grading shift, this is precisely the content of
Proposition~\ref{prop_dot_slide}.
\end{proof}

\begin{lem}[Duality for $U_n$]
The equation
\begin{equation} \label{eq_prop_Udual}
    \Gamma^{G}\left(\;\xy 0;/r.16pc/:
    (-9,8)*{}="1";
    (-3,8)*{}="2";
    (-9,-16);"1" **\dir{-};
    "1";"2" **\crv{(-9,14) & (-3,14)} ?(0)*\dir{<};
    (9,-8)*{}="1";
    (3,-8)*{}="2";
    (9,16);"1" **\dir{-};
    "1";"2" **\crv{(9,-14) & (3,-14)} ?(1)*\dir{>} ?(.05)*\dir{>};
    (-15,8)*{}="1";
    (3,8)*{}="2";
    (-15,-16);"1" **\dir{-};
    "1";"2" **\crv{(-15,20) & (3,20)} ?(0)*\dir{<};
    (15,-8)*{}="1";
    (-3,-8)*{}="2";
    (15,16);"1" **\dir{-};
    "1";"2" **\crv{(15,-20) & (-3,-20)} ?(.03)*\dir{>}?(1)*\dir{>};
    (0,0)*{\twoIu};
    (24,-9)*{ \bfit{n+4}};
    (-20,9)*{ \bfit{n}};
    \endxy \; \right)
    \quad =
    \quad    \Gamma^{G}\left(   \xy
    (0,0)*{\twoId};
    (6,0)*{\bfit{n+4}};
 \endxy\right)
 \quad =
    \quad
 \Gamma^{G}\left(\;
    \xy 0;/r.16pc/:
    (9,8)*{}="1";
    (3,8)*{}="2";
    (9,-16);"1" **\dir{-};
    "1";"2" **\crv{(9,14) & (3,14)} ?(0)*\dir{<};
    (-9,-8)*{}="1";
    (-3,-8)*{}="2";
    (-9,16);"1" **\dir{-};
    "1";"2" **\crv{(-9,-14) & (-3,-14)} ?(1)*\dir{>} ?(.05)*\dir{>};
    (15,8)*{}="1";
    (-3,8)*{}="2";
    (15,-16);"1" **\dir{-};
    "1";"2" **\crv{(15,20) & (-3,20)} ?(0)*\dir{<};
    (-15,-8)*{}="1";
    (3,-8)*{}="2";
    (-15,16);"1" **\dir{-};
    "1";"2" **\crv{(-15,-20) & (3,-20)} ?(0.03)*\dir{>} ?(1)*\dir{>};
    (0,0)*{\twoIu};
    (24,-9)*{ \bfit{n+4}};
    (-20,9)*{ \bfit{n}};
    \endxy\;\right)
\end{equation}
holds in $\GrGs$ for all $n \in \Z$.
\end{lem}

\begin{proof}
It suffices to consider the element $1 \otimes \xi^{\alpha}$.  We have
\begin{eqnarray}
 \text{$ \Gamma^{G}\left(\;
    \xy 0;/r.16pc/:
    (9,8)*{}="1";
    (3,8)*{}="2";
    (9,-16);"1" **\dir{-};
    "1";"2" **\crv{(9,14) & (3,14)} ?(0)*\dir{<};
    (-9,-8)*{}="1";
    (-3,-8)*{}="2";
    (-9,16);"1" **\dir{-};
    "1";"2" **\crv{(-9,-14) & (-3,-14)} ?(1)*\dir{>} ?(.05)*\dir{>};
    (15,8)*{}="1";
    (-3,8)*{}="2";
    (15,-16);"1" **\dir{-};
    "1";"2" **\crv{(15,20) & (-3,20)} ?(0)*\dir{<};
    (-15,-8)*{}="1";
    (3,-8)*{}="2";
    (-15,16);"1" **\dir{-};
    "1";"2" **\crv{(-15,-20) & (3,-20)} ?(0.03)*\dir{>} ?(1)*\dir{>};
    (0,0)*{\twoIu};
    (24,-9)*{ \bfit{n+4}};
    (-20,9)*{ \bfit{n}};
    \endxy\;\right) \left( 1 \otimes \xi^{\alpha} \right)$}
    &=&
    \sum_{j=0}^{k}(-1)^{\alpha-1} x_{j,n} \otimes Y_{\alpha-1-j,n+4}.
\end{eqnarray}
We can write this as a sum from $j=0$ to $j=\alpha-1$ using that $x_{j,n}=0$ for
$j>k$ and $Y_{\ell,n+4}=0$ for $\ell<0$.  Now using \eqref{eq_Yslide} we have
\begin{eqnarray}
 &=&
 \sum_{j=0}^{\alpha-1}\sum_{p=0}^{\alpha-1-p}(-1)^{\alpha-1-p}
 x_{j,n} \otimes Y_{\alpha-1-j-p,n+2} \xi^{p} \\
 & \refequal{\eqref{eq_dY}} &
 \sum_{p=0}^{\alpha-1}\xi^{\alpha-1-p} \otimes  \xi^{p}.
\end{eqnarray}
This is the same as $\Gamma^{G}(\hat{U}_n)$ defined in
Definition~\ref{def_nil_generator}.  The other identity is proved similarly.
\end{proof}

For the remaining identities it is helpful to compute
  \begin{eqnarray}
    \Gamma^G\left(\;\xy
  (4,8)*{\bfit{n}};
  (0,-2)*{\cbub{n-1+\alpha}};
 \endxy \; \right)  \maps  1 \to
 (-1)^{\alpha} \sum_{\ell=0}^{N-k} y_{\ell,n}Y_{\alpha-\ell,n},
\qquad
    \Gamma^G\left(\;\xy
  (4,8)*{\bfit{n}};
  (0,-2)*{\ccbub{-n-1+\alpha}};
 \endxy \;\right)  \maps  1 \to
 (-1)^{\alpha} \sum_{\ell=0}^{k}
  x_{\ell,n}X_{\alpha-\ell,n} \nn
  \end{eqnarray}
which follow immediately from Definition~\ref{def_biadjoint}.

\begin{lem}[Positive degree of closed bubbles]
For all $m \geq 0$ we have
\[
\Gamma^G\left(\; \xy
 (-12,0)*{\cbub{m}};
 (-8,8)*{\bfit{n}};
 \endxy \;\right) =0 \quad \text{if $m< n-1$}, \qquad \quad
 \Gamma^G\left(\; \xy
 (-12,0)*{\ccbub{m}};
 (-8,8)*{\bfit{n}};
 \endxy\; \right) = 0 \quad \text{if $m< -n-1$,}
\]
for all $n \in \Z$.
\end{lem}

\begin{proof}
This is clear from the definitions above and the positive degree of $x_{j,n}$,
$y_{j,n}$, $X_{j,n}$, and $Y_{j,n}$.
\end{proof}

\begin{lem}[Reduction to bubbles] The equations
\begin{eqnarray}
  \text{$\Gamma^{G}\left(\xy 0;/r.15pc/:
  (14,8)*{\bfit{n}};
  (0,0)*{\twoIu};
  (-3,-12)*{\bbsid};
  (-3,8)*{\bbsid};
  (3,8)*{}="t1";
  (9,8)*{}="t2";
  (3,-8)*{}="t1'";
  (9,-8)*{}="t2'";
   "t1";"t2" **\crv{(3,14) & (9, 14)};
   "t1'";"t2'" **\crv{(3,-14) & (9, -14)};
   (9,0)*{\bbf{}};
 \endxy\right)$} = \Gamma^{G}\left(\;-\sum_{\ell=0}^{-n}
   \xy 0;/r.18pc/:
  (14,8)*{\bfit{n}};
  (0,0)*{\bbe{}};
  (14,-2)*{\cbub{n-1+\ell}};
  (0,6)*{\bullet}+(5,-1)*{\scs \;\; -n-\ell};
 \endxy \;\;\right)\label{eq_lem_reductionI}
\qquad
  \text{$ \Gamma^{G}\left(\xy 0;/r.15pc/:
  (-12,8)*{\bfit{n}};
  (0,0)*{\twoIu};
  (3,-12)*{\bbsid};
  (3,8)*{\bbsid};
  (-9,8)*{}="t1";
  (-3,8)*{}="t2";
  (-9,-8)*{}="t1'";
  (-3,-8)*{}="t2'";
   "t1";"t2" **\crv{(-9,14) & (-3, 14)};
   "t1'";"t2'" **\crv{(-9,-14) & (-3, -14)};
   (-9,0)*{\bbf{}};
 \endxy\right)$} =
 \Gamma^{G}\left(\; \sum_{j=0}^{n}
   \xy 0;/r.18pc/:
  (-12,8)*{\bfit{n}};
  (0,0)*{\bbe{}};
  (-12,-2)*{\ccbub{-n-1+j}};
  (0,6)*{\bullet}+(5,-1)*{\scs n-j};
 \endxy\;\;\right) \label{eq_lem_reductionII} \nn
\end{eqnarray}
of bimodule maps hold in $\GrGs$ for all $n\in \Z$.
\end{lem}

\begin{proof}
Consider the first equality on $1 \in \HG_{k,k+1}$. We have
\begin{eqnarray}
 \text{$\Gamma^{G}\left(\xy 0;/r.15pc/:
  (14,8)*{\bfit{n}};
  (0,0)*{\twoIu};
  (-3,-12)*{\bbsid};
  (-3,8)*{\bbsid};
  (3,8)*{}="t1";
  (9,8)*{}="t2";
  (3,-8)*{}="t1'";
  (9,-8)*{}="t2'";
   "t1";"t2" **\crv{(3,14) & (9, 14)};
   "t1'";"t2'" **\crv{(3,-14) & (9, -14)};
   (9,0)*{\bbf{}};
 \endxy\right)$}
\maps 1 \;\; \to \;\;
 -\sum_{\ell=0}^{N-k}\sum_{j=0}^{N-k-\ell-1}(-1)^{-n-j} \xi^{j}
 y_{\ell,n}Y_{-n-\ell-j,n}.
\end{eqnarray}
But the $Y_{-n-\ell-j,n}$ are only nonzero when $j\leq -n-\ell=N-2k-\ell$.  For
$k>0$ this implies $-n-\ell \leq N-k-\ell-1$.  After changing the $j$-summation
to reflect this fact, so that $0 \leq j \leq -n-\ell$, the above is equal to the
image of $1 \in \HG_{k,k+1}$ under the bimodule map on the right hand side of the
first equality above. The second equality is proven similarly.
\end{proof}

\begin{lem}[NilHecke action] The equations
\begin{equation}
 \Gamma^G\left(\;\xy 0;/r.18pc/:
  (0,-8)*{\twoIu};
  (0,8)*{\twoIu};
  (8,8)*{\bfit{n}};
 \endxy\;\right)
 \;\;= \;\; 0, \qquad
 \Gamma^G\left(\;\vcenter{ \xy 0;/r.18pc/:
    (0,0)*{\twoIu};
    (6,16)*{\twoIu};
    (-3,8);(-3,24) **\dir{-}?(1)*\dir{>};
    (0,32)*{\twoIu};
    (9,-8);(9,8) **\dir{-};
    (9,24);(9,42) **\dir{-}?(1)*\dir{>};
    (14,16)*{\bfit{n}};
 \endxy} \; \right)
 \;\;
 =
 \;\;\Gamma^G\left(\;\;
  \vcenter{\xy 0;/r.18pc/:
    (0,0)*{\twoIu};
    (-6,16)*{\twoIu};
    (3,8);(3,24) **\dir{-}?(1)*\dir{>};
    (0,32)*{\twoIu};
    (-9,-8);(-9,8) **\dir{-};
    (-9,24);(-9,42) **\dir{-}?(1)*\dir{>};
    (8,16)*{\bfit{n}};
 \endxy}\;\right) , \nn
\end{equation}
\begin{eqnarray}
  \Gamma^G\left(\;\;\xy 0;/r.18pc/:
  (3,9);(3,-9) **\dir{-}?(.5)*\dir{<}+(2.3,0)*{};
  (-3,9);(-3,-9) **\dir{-}?(.5)*\dir{<}+(2.3,0)*{};
  (8,2)*{\bfit{n}};
 \endxy\right)
  \; = \;
  \Gamma^G\left(\xy 0;/r.18pc/:
  (0,0)*{\twoIu};
  (-2,-5)*{ \bullet};
  (8,2)*{\bfit{n}};
 \endxy\right)
  - \;
  \Gamma^G\left(\xy 0;/r.18pc/:
  (0,0)*{\twoIu};
  (2,5)*{ \bullet};
  (8,2)*{\bfit{n}};
 \endxy\right)
 \; = \;
  \Gamma^G\left(\xy 0;/r.18pc/:
  (0,0)*{\twoIu};
  (-2,5)*{ \bullet};
  (8,2)*{\bfit{n}};
 \endxy\right)
  - \;
 \Gamma^G\left( \xy 0;/r.18pc/:
  (0,0)*{\twoIu};
  (2,-5)*{ \bullet};
  (8,2)*{\bfit{n}};
 \endxy \right) \nn
\end{eqnarray}
hold in $\GrGs$ for all $n \in \Z$.
\end{lem}

\begin{proof}
The proof is identical to the proof of the same relation for the 2-functor
$\Gamma$ in \cite{Lau3}.  That proof never made use of the relations imposed on
the canonical generators for the ordinary cohomology rings of the iterated flag
varieties.
\end{proof}

\begin{lem}[Identity decomposition]
The equations
\begin{eqnarray}
\Gamma^G\left(\;\vcenter{\xy 0;/r.18pc/:
  (-8,0)*{};
  (8,0)*{};
  (-4,10)*{}="t1";
  (4,10)*{}="t2";
  (-4,-10)*{}="b1";
  (4,-10)*{}="b2";
  "t1";"b1" **\dir{-} ?(.5)*\dir{<};
  "t2";"b2" **\dir{-} ?(.5)*\dir{>};
  (14,6)*{\bfit{n}};
  \endxy}\;\right)
\quad = \quad
 \Gamma^G\left(\; -\;\;\vcenter{\xy 0;/r.18pc/:
  (0,0)*{\FEtEF};
  (0,-10)*{\EFtFE};
  (14,2)*{\bfit{n}};
  \endxy}\;\right)
  \quad + \quad
   \text{$\Gamma^G\left(\; \sum_{\ell=0}^{n-1} \sum_{j=0}^{\ell}
    \vcenter{\xy 0;/r.18pc/:
    (-8,0)*{};
  (8,0)*{};
  (-4,-15)*{}="b1";
  (4,-15)*{}="b2";
  "b2";"b1" **\crv{(5,-8) & (-5,-8)}; ?(.1)*\dir{<} ?(.9)*\dir{<}
  ?(.8)*\dir{}+(0,-.1)*{\bullet}+(-5,2)*{\scs \ell-j};
  (-4,15)*{}="t1";
  (4,15)*{}="t2";
  "t2";"t1" **\crv{(5,8) & (-5,8)}; ?(.15)*\dir{>} ?(.9)*\dir{>}
  ?(.4)*\dir{}+(0,-.2)*{\bullet}+(3,-2)*{\scs n-1-\ell};
  (0,0)*{\ccbub{\scs -n-1+j}};
  (16,6)*{\bfit{n}};
  \endxy}\;\right)$}
  \nn\\ \nn \\
  \Gamma^G\left(\; \vcenter{\xy 0;/r.18pc/:
  (-8,0)*{};
  (8,0)*{};
  (14,6)*{\bfit{n}};
  (-4,10)*{}="t1";
  (4,10)*{}="t2";
  (-4,-10)*{}="b1";
  (4,-10)*{}="b2";
  "t1";"b1" **\dir{-} ?(.5)*\dir{>};
  "t2";"b2" **\dir{-} ?(.5)*\dir{<};
  \endxy}\;\right)
\quad = \quad
 \Gamma^G\left(\;-\;\;
 \vcenter{\xy 0;/r.18pc/:
  (0,0)*{\EFtFE};
  (0,-10)*{\FEtEF};
  (14,2)*{\bfit{n}};
  \endxy}\;\right)
  \quad + \quad
\text{$\Gamma^G\left(\;\sum_{\ell=0}^{-n-1} \sum_{j=0}^{\ell}
    \vcenter{\xy 0;/r.18pc/:
    (-8,0)*{};
  (8,0)*{};
  (-4,-15)*{}="b1";
  (4,-15)*{}="b2";
  "b2";"b1" **\crv{(5,-8) & (-5,-8)}; ?(.1)*\dir{>} ?(.95)*\dir{>}
  ?(.8)*\dir{}+(0,-.1)*{\bullet}+(-5,2)*{\scs \ell-j};
  (-4,15)*{}="t1";
  (4,15)*{}="t2";
  "t2";"t1" **\crv{(5,8) & (-5,8)}; ?(.15)*\dir{<} ?(.8)*\dir{<}
  ?(.4)*\dir{}+(0,-.2)*{\bullet}+(3,-2)*{\scs -n-1-\ell};
  (0,0)*{\cbub{\scs n-1+j}};
  (16,6)*{\bfit{n}};
  \endxy}\;\right)$} \nn
\end{eqnarray}
hold for all $n \in \Z$.
\end{lem}

\begin{proof}
Using Lemma~\ref{lem_biadjoint} proving the first identity is equivalent to
proving
\begin{equation} \label{eq_alternativeid}
 \Gamma\left(\; \xy 0;/r.18pc/:
  (-8,0)*{};
  (8,0)*{};
  (-4,-11)*{}="b1";
  (4,-11)*{}="b2";
  "b2";"b1" **\crv{(5,-2) & (-5,-2)}; ?(.15)*\dir{>} ?(.9)*\dir{>};
  (-4,11)*{}="t1";
  (4,11)*{}="t2";
  "t2";"t1" **\crv{(5,2) & (-5,2)}; ?(.15)*\dir{<} ?(.8)*\dir{<};
  (16,0)*{\bfit{n-2}};
  \endxy\;\right)
  \quad = \quad
 \Gamma\left(\; -
  \xy 0;/r.18pc/:
  (20,0)*{\bfit{n-2}};
  (-6,0)*{\twoId};
  (6,0)*{\twoIu};
  (9,-12)*{\bbsid};
  (-9,-12)*{\bbsid};
  (9,10)*{\bbsid};
  (-9,10)*{\bbsid};
  (-3,8)*{}="t1";
  (3,8)*{}="t2";
  "t1";"t2" **\crv{(-3,14) & (3, 14)};
  (-3,-8)*{}="t1";
  (3,-8)*{}="t2";
  "t1";"t2" **\crv{(-3,-14) & (3, -14)} ?(.1)*\dir{>} ?(1)*\dir{>};
 \endxy\;\right)
 \;\; + \;
 \Gamma\left(\; \sum_{\ell=0}^{n-1}
 \sum_{j=0}^{\ell}
    \xy
  (20,0)*{\bfit{n-2}};
  (12,0)*{\bbe{}};
  (0,-3)*{\ccbub{-n-1+j}};
  (-12,6)*{\bullet}+(6,1)*{\scs n-1-\ell};
  (12,6)*{\bullet}+(-4,1)*{\scs \ell-j};
  (-12,0)*{\bbf{}};
 \endxy\;\right)
\end{equation}
We compute the above maps on the elements $1 \otimes \xi^{\alpha}$ since these
elements together with relations of Proposition~\ref{prop_relation} and the
bimodule property determine the image on all other elements.

We begin by computing the image of the element $1 \otimes \xi^{\alpha}$ under the
maps on the right hand side of \eqref{eq_alternativeid}.
\begin{eqnarray}
 \Gamma\left(\xy0;/r.14pc/:
 (0,0)*{
   \xy 0;/r.18pc/:
  (20,0)*{\bfit{n-2}};
  (-6,0)*{\twoId};
  (6,0)*{\twoIu};
  (9,-12)*{\bbsid};
  (-9,-12)*{\bbsid};
  (9,10)*{\bbsid};
  (-9,10)*{\bbsid};
  (-3,8)*{}="t1";
  (3,8)*{}="t2";
  "t1";"t2" **\crv{(-3,14) & (3, 14)};
  (-3,-8)*{}="t1";
  (3,-8)*{}="t2";
  "t1";"t2" **\crv{(-3,-14) & (3, -14)} ?(.1)*\dir{>} ?(1)*\dir{>};
 \endxy};
 \endxy \right) \left(
 1 \otimes \xi^{\alpha}
    \right) = \sum_{s=0}^k
 \sum_{j=0}^{\alpha-1}
 \sum_{p=0}^{k-s-1}
 (-1)^{\alpha+n-p-j}
  \xi^{p}X_{n+\alpha-s-p-j-1,n} \otimes \xi^{j} x_{s,n-2}\nn \\
  \quad \qquad +\quad
    \sum_{s=1}^k
    \sum_{j=0}^{\alpha}
    \sum_{p=0}^{k-s-1}
 (-1)^{\alpha+n-p-j+1}
 \xi^{p}X_{n+\alpha-s-p-j,n} \otimes \xi^{j} x_{s-1,n}
    \nn ~.
\end{eqnarray}

Removing terms that are zero and shifting the second term by letting $s'=s-1$ we
have
\begin{eqnarray}
(-1)^{\alpha+n}\left( \sum_{s=0}^{k-1}
 \sum_{j=0}^{\alpha-1}
 \sum_{p=0}^{k-s-1}
 (-1)^{-p-j}
  \xi^{p}X_{n+\alpha-s-p-j-1,n} \otimes \xi^{j} x_{s,n-2}
   \right.  \hspace{1in}
    \nn \\
     \hspace{1.6in}
     \left.
     -
    \sum_{s'=0}^{k-1}
    \sum_{j=0}^{\alpha}
    \sum_{p=0}^{k-s'-2}
     (-1)^{-p-j}
 \xi^{p}X_{n+\alpha-s'-p-j-1,n} \otimes \xi^{j} x_{s',n-2}
    \right) .\nn
\end{eqnarray}
Now the only terms that do not cancel are the $p=k-s-1$ term of the first factor
and the $j=\alpha$ term of the second factor:
\begin{eqnarray}
 (-1)^{\alpha+k-N}
 \sum_{s=0}^{k-1}
 \sum_{j=0}^{\alpha-1}
 (-1)^{-j+1}
  \xi^{k-s-1}X_{\alpha+k-N-j,n} \otimes \xi^{j} x_{s,n-2} \hspace{1.5in} \nn\\
     -
    \sum_{s'=0}^{k-2}
    \sum_{p=0}^{k-s'-2}
     (-1)^{n-p}
\xi^{p}X_{n-s'-p-1,n} \otimes \xi^{\alpha} x_{s',n-2} .\hspace{1in}
\end{eqnarray}
In the first term note that $\alpha-1 \geq \alpha-(N-k)$ since $N-k \geq 1$,
otherwise $k=N$. So we change the upper limit of the $j$ summand to
$\alpha-(N-k)$ and apply \eqref{eq_Xslide}.  In the second summand we note that
we must have $n-s'-p-1 \geq 0$ and that $k-s-2 \geq n-s-1$ if $N-k \geq 1$.
Hence, the $s'$ summation goes only as high as $n-1$, and the $p$ summation only
as high as $n-1-s$ yielding
\begin{eqnarray}
 (-1)^{\alpha+k-N+1}
 \sum_{s=0}^{k-1}
 (-1)^{s}
   \xi^{k-s-1} \otimes X_{\alpha+k-N,n-2} x_{s,n-2}
     -
    \sum_{s'=0}^{n-1}
    \sum_{p=0}^{n-1-s'}
     (-1)^{n-p}
 \xi^{p}X_{n-s-p-1,n} \otimes \xi^{\alpha} x_{s,n-2}.\nn\\ \label{eq_bigmap}
\end{eqnarray}

Now we compute the other maps involved in \eqref{eq_alternativeid}. The second
map on the right hand side is
\begin{eqnarray}
 \Gamma\left(\sum_{\ell=0}^{n-1}
 \sum_{j=0}^{\ell}
    \xy 0;/r.18pc/:
  (20,0)*{\bfit{n-2}};
  (12,0)*{\bbe{}};
  (0,-3)*{\ccbub{-n-1+j}};
  (-12,6)*{\bullet}+(7,1)*{\scs n-1-\ell};
  (12,6)*{\bullet}+(3,1)*{\scs \quad \ell-j};
  (-12,0)*{\bbf{}};
 \endxy\right)\left(
1 \otimes \xi^{\alpha} \right) =
 \sum_{\ell=0}^{n-1}
 \sum_{j=0}^{\ell}
 \sum_{s=0}^{\min(j,k)}
 (-1)^j
 \xi^{n-1-\ell} x_{s,n}X_{j-s,n} \otimes \xi^{\ell-j+\alpha} .
    \nn
 \end{eqnarray}
Notice that $j\leq k$ since $j\leq n-1 = 2k-N-1 = k+(k-N-1)$ and $k-N-1< 0$.
Thus, we remove the $\min(j,k)$ in the summation and use
\eqref{eq_noncanonical_gen} to change the non-canonical generator $x_{s,n}$ into
canonical generators,
\begin{eqnarray}
\sum_{\ell=0}^{n-1}
 \sum_{j=0}^{\ell}
 \sum_{s=0}^{j}
 (-1)^j
 \xi^{n-1-\ell} X_{j-s,n} \otimes \xi^{\ell-j+\alpha}x_{s,n-2}
- \sum_{\ell=0}^{n-1}
 \sum_{j=0}^{\ell}
 \sum_{s=0}^{j}
 (-1)^j
 \xi^{n-1-\ell} X_{j-s,n} \otimes \xi^{\ell-j+\alpha+1}x_{s-1,n-2} . \nn
\end{eqnarray}
Now shift the indices of the second term by letting $s'=s-1$ and $j'=j-1$ so that
the two terms cancel leaving only the $j=\ell$ term of the first term,
\begin{eqnarray}
\sum_{\ell=0}^{n-1}
 \sum_{s=0}^{\ell}
 (-1)^{\ell}
 \xi^{n-1-\ell} X_{\ell-s,n} \otimes \xi^{\alpha}x_{s,n-2}~.
\end{eqnarray}
If we let $p=n-1-\ell$ and $s=s'$ this term cancels with the second term in
\eqref{eq_bigmap}.  One can check that minus the first term in \eqref{eq_bigmap}
is equal to the image of $1 \otimes \xi^{\alpha}$ under the left hand side of
\eqref{eq_alternativeid}.
\end{proof}

\begin{thm}\label{thm_flag}
The assignments given in subsection~\ref{subsec_define_gamma} define a graded
additive 2-functor $\Gamma_N^G \maps \Ucatq \to \GrGs$.  By restricting to degree
preserving 2-morphisms we also get an additive 2-functor $\Gamma_N^G \maps \Ucat
\to \GrG$.
\end{thm}

\begin{proof}
We have already seen that $\Gamma_N^G$ preserves composites of 1-morphisms up to
isomorphism. The lemmas above show that $\Gamma_N^G$ preserve the defining
relations of the 2-morphisms in $\Ucatq$.  Therefore, since degrees and direct
sums are also preserved, $\Gamma_N^G$ is a graded additive 2-functor.  It is
clear that restricting to the degree preserving maps gives the restricted
additive 2-functor $\Gamma_N^G \maps \Ucat \to \GrG$.
\end{proof}

%
\subsection{Categorification of $V_N$}
%

\begin{thm} \label{thm_catVN}
The representation $\Gamma^{G}_N \maps \Ucat \to \GrG$ yields a representation
$\dot{\Gamma}_N^G \maps \UcatD \to \GrG$. This representation categorifies the irreducible $(N+1)$-dimensional representation
$V_N$ of $\U$.
\end{thm}

\begin{proof}
Idempotent 2-morphisms split in $\GrG$ so, by the universal property of the
Karoubi envelope, we have
\[
 \xymatrix{
  \Ucat \ar[dr]_{\Gamma_N^{G}} \ar[r]^{} & \UcatD \ar@{.>}[d]^{\dot{\Gamma}_N^G} \\
  & \GrG
 }
\]
showing that $\Gamma_N^G$ induces a representation $\dot{\Gamma}_N^G$ of
$\UcatD$.

The rings $\HG_k$ forming the objects of $\GrGs$ are (positively, even) graded
local rings with degree zero part isomorphic to $\Q$. Thus, every graded
finitely-generated projective module is free, and $\HG_k$ has (up to isomorphism
and grading shift) a unique graded indecomposable projective module. Let
$\HG_k{\rm -pmod}$ denote the category of finitely generated graded projective
$\HG_k$-modules.  The split Grothendieck group of the category
$\bigoplus_{j=0}^{N}\HG_j{\rm -pmod}$ is then a free $\Z[q,q^{-1}]$-module of
rank $N+1$, freely generated by the indecomposable projective modules, where
$q^i$ acts by shifting the grading degree by $i$. Thus, we have
\begin{eqnarray}
K_0\big(\bigoplus_{k=0}^{N}H_k{\rm -pmod}\big) \cong {}_{\cal{A}}(V_N), \qquad
K_0\big(\bigoplus_{k=0}^{N}H_k{\rm -pmod}\big)\otimes_{\Z[q,q^{-1}]}\Q(q) \cong
V_N,
\end{eqnarray}
as $\Z[q,q^{-1}]$-modules, respectively $\Q(q)$-modules, where
${}_{\cal{A}}(V_N)$ is a representation of $\UA$, an integral form of the
representation $V_N$ of $\U$.

The bimodules $\Gamma_N^G(\onen)$, $\Gamma_N^G(\cal{E}\onen)$ and
$\Gamma_N^G(\cal{F}\onen)$ induce, by tensor product, functors on the graded
module categories.  More precisely, consider the restriction functors
\begin{eqnarray}
\Res^{k,k+1}_{k} &\maps& \HG_{k,k+1}{\rm -pmod} \to \HG_{k}{\rm -pmod} \nn\\
\Res^{k,k+1}_{k+1}& \maps& \HG_{k,k+1}{\rm -pmod} \to \HG_{k+1}{\rm -pmod} \nn
\end{eqnarray} given
by the inclusions $\HG_k \to \HG_{k,k+1}$ and $\HG_{k+1} \to \HG_{k,k+1}$. For $0
\leq k \leq N$ and $n=2k-N$ define functors
\begin{eqnarray}
 \mathbf{1}_n &:=& \HG_k \otimes_{\HG_k} \maps \HG_{k} {\rm -pmod} \to \HG_{k}{\rm -pmod} \\
 \mathbf{E1}_n &:=& \Res^{k,k+1}_{k+1} \HG_{k+1,k} \otimes_{\HG_k} \{1-N+k\} \maps
 \HG_{k}{\rm -pmod} \to \HG_{k+1}{\rm -pmod} \\
 \mathbf{F1}_{n+2} &:=& \Res^{k,k+1}_{k} \HG_{k,k+1} \otimes_{\HG_{k+1}} \{-k\} \maps
\HG_{k+1} {\rm -pmod} \to \HG_{k} {\rm -pmod}.
\end{eqnarray}
These functors have both left and right adjoints and commute with the shift
functor, so they will induce $\Z[q,q^{-1}]$-module maps on Grothendieck groups.
Furthermore, the 2-functor $\dot{\Gamma}_N^G$ must preserve the relations of
$\UcatD$, so by Theorem~\ref{thm_Groth} these functors satisfy relations lifting
those of $\U$.
\end{proof}


%
%

%
%

\def\cprime{$'$}

%
%

\newenvironment{hpabstract}{%
  \renewcommand{\baselinestretch}{0.2}
  \begin{footnotesize}%
}{\end{footnotesize}}%
\newcommand{\hpeprint}[1]{%
  \href{http://arXiv.org/abs/#1}{\texttt{#1}}}%
\newcommand{\hpspires}[1]{%
  \href{http://www.slac.stanford.edu/spires/find/hep/www?#1}{\ (spires)}}%
\newcommand{\hpmathsci}[1]{%
  \href{http://www.ams.org/mathscinet-getitem?mr=#1}{\texttt{MR #1}}}%

%

%
\end{document}